\documentclass[letterpaper,10 pt,conference]{ieeeconf}  

\IEEEoverridecommandlockouts                              

\usepackage{graphicx,algorithm,algpseudocode,amsmath,amsfonts,verbatim,mathtools,enumerate,bm,mathabx,hhline,multirow,tikz,braket,multirow,adjustbox,aircraftshapes,breakcites,amssymb,stfloats,amsthm}

\usepackage[style=base]{caption}
\usepackage{subcaption}
\usetikzlibrary{patterns}
\usetikzlibrary{calc}
\usepackage[flushleft]{threeparttable}

\allowdisplaybreaks

\newcommand{\diag}{\mathop{\rm diag}}

\newcommand{\prox}{\mathrm{prox}}
\newcommand{\rec}{\mathrm{rec\,}}

\newcommand{\argmin}{\mathop{\rm argmin}}

\newcommand{\norm}[1]{\left\lVert#1\right\rVert}
\newcommand{\mnorm}[1]{{\left\vert\kern-0.25ex\left\vert\kern-0.25ex\left\vert #1 
    \right\vert\kern-0.25ex\right\vert\kern-0.25ex\right\vert}}

\newtheorem{theorem}{Theorem}
\newtheorem{lemma}{Lemma}

\newtheorem{remark}{Remark}
\newtheorem{proposition}{Proposition}
\newtheorem{assumption}{Assumption}

\newcommand{\eg}{{\it e.g.}}
\newcommand{\ie}{{\it i.e.}}

\title{\LARGE \bf Proportional-Integral Projected Gradient Method for Infeasibility Detection in Conic Optimization
}

\author{Yue~Yu and Ufuk~Topcu
\thanks{The authors are with the Oden Institute for Computational Engineering and Sciences, The University of Texas at Austin, Austin, TX, 78712, USA (e-mails:  yueyu@utexas.edu,utopcu@utexas.edu).}
}

\begin{document}

\maketitle
\thispagestyle{empty}
\pagestyle{empty}

\begin{abstract} 
A constrained optimization problem is primal infeasible if its constraints cannot be satisfied, and dual infeasible if the constraints of its dual problem cannot be satisfied. We propose a novel iterative method, named proportional-integral projected gradient method (PIPG), for detecting primal and dual infeasiblity in convex optimization with quadratic objective function and conic constraints. The iterates of PIPG either asymptotically provide a proof of primal or dual infeasibility, or asymptotically satisfy a set of primal-dual optimality conditions. Unlike existing methods, PIPG does not compute matrix inverse, which makes it better suited for large-scale and real-time applications. We demonstrate the application of PIPG in quasiconvex and mixed-integer optimization using examples in constrained optimal control.
\end{abstract}

\section{Introduction}
Conic optimization is the minimization of a quadratic objective function subject to conic constraints:
\begin{equation}\label{opt: conic}
    \begin{array}{ll}
    \underset{z}{\mbox{minimize}} & \frac{1}{2} z^\top P z+q^\top z \\
    \mbox{subject to} & Hz-g\in\mathbb{K}, \enskip z\in \mathbb{D},
    \end{array}
\end{equation}
where \(z\in\mathbb{R}^n\) is the optimization variable, symmetric positive semidefinite matrix \(P\in\mathbb{R}^{n\times n}\) and vector \(q\in\mathbb{R}^n\) define the quadratic objective function, matrix \(H\in\mathbb{R}^{m\times n}\) and vector \(g\in\mathbb{R}^m\) are constraint parameters, cone \(\mathbb{K}\subseteq\mathbb{R}^m\) and set \(\mathbb{D}\subseteq\mathbb{R}^n\) are nonempty, closed, and convex. Optimization \eqref{opt: conic} includes many common convex optimization problems as special cases, including linear,  quadratic, second-order cone, and semidefinite programing \cite{ben2001lectures,boyd2004convex}.

Optimization \eqref{opt: conic} is \emph{primal infeasible} if there exists no \(z\in\mathbb{D}\) such that \(Hz-g\in\mathbb{K}\), it is \emph{dual infeasible} if the constraints of its dual problem cannot be satisfied. If optimization \eqref{opt: conic} is not primal infeasible, the infeasibility of its dual problem implies that the primal optimal value is unbounded from below \cite{banjac2019infeasibility}.

Given optimization \eqref{opt: conic}, infeasibility detection is the problem of providing a proof of primal or dual infeasibility if it is the case \cite{banjac2019infeasibility}. In particular, a proof of primal infeasibility is the existence of a hyperplane that separates cone \(\mathbb{K}\) from the set \(\{Hz-g|z\in\mathbb{D}\}\). A proof of dual infeasibility is the existence of a vector \(\overline{z}\in\mathbb{R}^n\) that satisfies the following conditions for all \(z\in\mathbb{D}\) \cite{banjac2021minimal}:
\[P\overline{z}=0, \enskip q^\top \overline{z}<0, \enskip H\overline{z}\in\mathbb{K},\enskip  z+\overline{z}\in\mathbb{D}.\]
Infeasibility detection is necessary for  adjusting the parameters in pathological optimization problems \cite{barratt2021automatic}, and an integral subproblem in quasiconvex optimization \cite{boyd2004convex,agrawal2020disciplined} and mixed-integer optimization \cite{conforti2014integer,wolsey2020integer}.

Infeasibility detection via Douglas-Rachford splitting method (DRS) has recently attracted increasing attention \cite{raghunathan2014infeasibility,o2016conic,liu2019new,banjac2019infeasibility}. DRS detects infeasibility of optimization \eqref{opt: conic} by computing either a nonzero solution of the \emph{homogeneous self-dual embedding} \cite{o2016conic}, or the \emph{minimal-displacement vector} \cite{raghunathan2014infeasibility,liu2019new,banjac2019infeasibility}. Compared with traditional infeasibility detection methods \cite{ye1994nl,nesterov1999infeasible}, DRS  has empirically achieved up to three-orders-of-magnitude speedups in computation for relatively large-scale problems \cite{o2016conic}.

A challenge of DRS-based infeasibility detection is computing matrix inverse \cite{o2016conic,raghunathan2014infeasibility,liu2019new,banjac2019infeasibility}. Such computation becomes expensive as the size of optimization \eqref{opt: conic} increases, or when optimization \eqref{opt: conic} is solved repeatedly with different constraint parameters in real-time \cite{malyuta2021convex,malyuta2021advances}. While such matrix inverse is not necessary for problems with only linear constraints \cite{malitsky2017primal,applegate2021infeasibility}, whether it is necessary in general for optimization \eqref{opt: conic} is, to our best knowledge, still an open question. 

We propose a novel infeasibility detection method for optimization \eqref{opt: conic}, named \emph{proportional-integral projected gradient method (PIPG)}. The iterates of PIPG either asymptotically satisfy a set of primal-dual optimality conditions or diverge. In the latter case, the difference between the consecutive iterates converges to a nonzero vector, known as the minimal displacement vector, that proves primal or dual infeasibility.

PIPG is the first infeasibility detection method for optimization \eqref{opt: conic} that avoids computing matrix inverse. All existing methods compute matrix inverse as a subroutine of either interior point methods \cite{ye1994nl,nesterov1999infeasible} or DRS \cite{raghunathan2014infeasibility,o2016conic,liu2019new,banjac2019infeasibility}. In contrast, PIPG only computes matrix multiplication and projections onto the set \(\mathbb{D}\) and the cone \(\mathbb{K}\) in optimization \eqref{opt: conic}. As a result, PIPG allows implementation for large-scale and real-time problems  using limited computational resources.

PIPG provides an efficient method for quasiconvex and mixed-integer optimization problems. We demonstrate the application of PIPG in these problems using two examples in constrained optimal control: minimum-time landing and nonconvex path-planning of a quadrotor.

The rest of the paper is organized as follows. After some basic results in convex analysis and monotone operator theory, Section~\ref{sec: preliminary} reviews the existing results on DRS-based infeasibility detection. Section~\ref{sec: PIPG} introduces PIPG along with its convergence properties. Section~\ref{sec: experiment} demonstrates the application of PIPG in constrained optimal control problems. Finally, Section~\ref{sec: conclusions} concludes and comments on future work directions. 
\section{Preliminaries and related work}
\label{sec: preliminary}

We will review some basic results in convex analysis and monotone operator theory, and some existing results on infeasibilty detection using Douglas-Rachford splitting method. 

\subsection{Notations and preliminaries}
\label{subsec: preliminaries}
We let \(\mathbb{N}\), \(\mathbb{R}\) and \(\mathbb{R}_+\) denote the set of positive integer, real numbers, and non-negative real numbers, respectively. For a vector \(z\in\mathbb{R}^n\) and integers \(i, j\in\mathbb{N}\) (\( 1\leq i< j\leq n \)), we let \([z]_j\) denote the \(j\)-th element of vector \(z\), and \([z]_{i:j}\in\mathbb{R}^{j-i+1}\) denote the vector composed of the entries of \(z\) between (and including) the \(i\)-th entry and the \(j\)-th entry. For two vectors \(z, z'\in\mathbb{R}^n\), we let \(\langle z, z'\rangle\) denote their inner product, and \(\norm{z}\coloneqq \sqrt{\langle z, z\rangle}\) denote the \(\ell_2\) norm of \(z\). We let \(I_n\) denote the \(n\times n\) identity matrix and \(0_{m\times n}\) denote the \(m\times n\) zero matrix. When the dimensions of these matrices are clear from the context, we omit their subscripts. For a matrix \(H\in\mathbb{R}^{m\times n}\), \(H^\top\) denotes its transpose, \(\mnorm{H}\) denotes its largest singular value. For a square matrix \(A\in\mathbb{R}^{n\times n}\), \(\exp(A)\) denotes the matrix exponential of \(A\). For a set \(\mathbb{S}\), we let \(-\mathbb{S}\coloneqq\{z|-z\in\mathbb{S}\}\). Given two sets \(\mathbb{S}_1\) and \(\mathbb{S}_2\), we let \(\mathbb{S}_1\times\mathbb{S}_2\) denote their Cartesian product. We say function \(f:\mathbb{R}^n\to\mathbb{R}\cup \{+\infty\}\) is convex if \(f(\alpha z+(1-\alpha)z')\leq \alpha f(z)+(1-\alpha)f(z')\)
for all \(\alpha\in[0, 1]\). We say set \(\mathbb{D}\subseteq\mathbb{R}^n\) is convex if \(\alpha z+(1-\alpha)z'\in\mathbb{D}\) for any \(\alpha\in[0, 1]\) and \(z, z'\in\mathbb{D}\). We say set \(\mathbb{K}\subseteq\mathbb{R}^m\) is a convex cone if \(\mathbb{K}\) is a convex set and \(\alpha w\in\mathbb{K}\) for any \(w\in\mathbb{K}\) and \(\alpha\in\mathbb{R}_+\).

Let \(f:\mathbb{R}^n\to\mathbb{R}\cup \{+\infty\}\) be a convex function. The \emph{proximal operator} of function \(f\) is a function given by
\begin{equation}
    \prox_{f}(z)\coloneqq \underset{z'}{\argmin}\,\, f(z')+\frac{1}{2}\norm{z'-z}^2
\end{equation}
for all \(z\in\mathbb{R}^n\). Let \(\mathbb{D}\subseteq\mathbb{R}^n\) be a nonempty closed convex set. The \emph{projection} of \(z\in\mathbb{R}^n\) onto set \(\mathbb{D}\) is given by
\begin{equation}\label{eqn: projection}
\pi_{\mathbb{D}}[z]\coloneqq \underset{z'\in\mathbb{D}}{\argmin} \norm{z-z'}.
\end{equation}
The \emph{point-to-set distance} from \(z\) to set \(\mathbb{D}\) is given by
\begin{equation}\label{eqn: distance func}
    d(z|\mathbb{D})\coloneqq \underset{z'\in\mathbb{D}}{\min} \norm{z-z'}.
\end{equation}
The \emph{normal cone} of set \(\mathbb{D}\) at \(z\) is given by
\begin{equation}\label{eqn: normal cone}
    N_{\mathbb{D}}(z)\coloneqq\{y\in\mathbb{R}^n|\langle y, z'-z\rangle\leq 0, \forall z'\in\mathbb{D}\}.
\end{equation}
The \emph{recession cone} of set \(\mathbb{D}\) is a nonempty closed convex cone given by
\begin{equation}\label{eqn: rec cone}
    \rec \mathbb{D}\coloneqq \{y\in\mathbb{R}^n| y+z\in\mathbb{D}, \forall z\in\mathbb{D}\}.
\end{equation}
The \emph{indicator function} of set \(\mathbb{D}\) is defined as
\begin{equation}
    \delta_{\mathbb{D}}(z)\coloneqq \begin{cases}
    0, & \text{if } z\in\mathbb{D},\\
    +\infty, & \text{otherwise.}
    \end{cases}
\end{equation}
for all \(z\in\mathbb{R}^n\).
The \emph{support function} of set \(\mathbb{D}\) is given by
\begin{equation}\label{eqn: support}
    \sigma_{\mathbb{D}}(z)\coloneqq \sup_{y\in\mathbb{D}}\,\langle y, z\rangle
\end{equation}
for all \(z\in\mathbb{R}^n\).
Let \(\mathbb{K}\subseteq\mathbb{R}^m\) be a nonempty closed convex cone. The recession cone of cone \(\mathbb{K}\) is itself, \ie, \(\rec \mathbb{K}=\mathbb{K}\) \cite[Cor. 6.50]{bauschke2017convex}.
The \emph{polar cone} of \(\mathbb{K}\) is a closed convex cone given by
\begin{equation}\label{eqn: polar cone}
    \mathbb{K}^\circ\coloneqq \{w\in\mathbb{R}^m| \langle w, y\rangle\leq 0, \forall y\in\mathbb{K}\}.
\end{equation}
One can verify that \((\mathbb{K}^\circ)^\circ=\mathbb{K}\) \cite[Cor. 6.21]{rockafellar2009variational}. 
We will use the following results on polar cone.
\begin{lemma}\cite[Thm. 6.30]{bauschke2017convex}\label{lem: Moreau}
If \(\mathbb{K}\subset\mathbb{R}^m\) is a nonempty closed convex cone, then \(w=\pi_{\mathbb{K}}[w]+\pi_{\mathbb{K}^\circ}[w]\) and \(\langle\pi_{\mathbb{K}}[w], \pi_{\mathbb{K}^\circ}[w]\rangle=0\) for all \(w\in\mathbb{R}^m\).
\end{lemma}

We say function \(T:\mathbb{R}^p\to\mathbb{R}^p\) is a \emph{\(\gamma\)-averaged operator} for some \(\gamma\in(0, 1)\) if and only if the following condition holds for all \(\xi_1, \xi_2\in\mathbb{R}^p\) \cite[Prop. 4.35]{bauschke2017convex}:
\begin{equation}\label{eqn: def avg op}
\begin{aligned}
     &\norm{T(\xi_1)-T(\xi_2)}^2-\norm{\xi_1-\xi_2}^2\\
     &\leq \frac{\gamma-1}{\gamma}\norm{\xi_1-T(\xi_1)-\xi_2+T(\xi_2)}^2.
\end{aligned}
\end{equation}
We will use the following result on averaged operators.
\begin{lemma}\cite[Lem. 5.1]{banjac2019infeasibility}
\label{lem: min dis}
Let \(T:\mathbb{R}^p\to\mathbb{R}^p\) be a \(\gamma\)-averaged operator for some \(\gamma\in(0, 1)\). Let \(\xi^1\in\mathbb{R}^p\) and \(\xi^{j+1}\coloneqq T(\xi^j)\) for all \(j\in\mathbb{N}\). Then there exists \(\overline{\xi}\in\mathbb{R}^p\), known as the \emph{minimal-displacement vector} of operator \(T\), such that
\[\lim_{j\to\infty} \frac{\xi^j}{j}=\lim_{j\to\infty}\xi^{j+1}-\xi^j=\overline{\xi}.\]
\end{lemma}

\subsection{Infeasibility detection via Douglas-Rachford splitting method}

Douglas-Rachford splitting method (DRS) is an iterative method for optimization problems of the following form:
\begin{equation}\label{opt: DRS}
    \begin{array}{ll}
        \underset{\xi}{\mbox{minimize}} & \ell(\xi)+r(\xi),
    \end{array}
\end{equation}
where \(\ell:\mathbb{R}^p\to\mathbb{R}\cup\{+\infty\}\) and \(r:\mathbb{R}^p\to\mathbb{R}\cup\{+\infty\}\) are convex functions. DRS uses the following iteration, where \(j\in\mathbb{N}\) is the iteration counter and \(\alpha\in(0, 2)\) is the step size:
\begin{equation}\label{eqn: DRS iter}
    \xi^{j+1}=\xi^j+\alpha(\prox_{\ell}(2\prox_{r}(\xi^j)-\xi^j)-\prox_{r}(\xi^j)).
\end{equation}
One can show that equation \eqref{eqn: DRS iter} is equivalent to \(\xi^{j+1}=T(\xi^j)\) for some \(\frac{\alpha}{2}\)-averaged operator \(T\) \cite{giselsson2016linear}. 

In the following, we will review two different ways of infeasibilty detection using DRS. For simplicity, we assume that DRS is terminated when a maximum number of iterations, denoted by \(k\) is reached. We also let \(\epsilon\in\mathbb{R}_+\) be such that numbers within the interval \([-\epsilon, \epsilon]\) are treated as approximately zero.
 
\subsubsection{Homogeneous self-dual embedding method}
The first way of infeasibility detection via DRS considers the following special case of optimization \eqref{opt: conic}:
\begin{equation}
    \begin{array}{ll}
    \underset{z}{\mbox{minimize}} & q^\top z \\
    \mbox{subject to} & Hz-g\in\mathbb{K}.
    \end{array}\label{opt: linear}
\end{equation}

The homogeneous self-dual embedding of optimization \eqref{opt: linear} is given by the following set of linear equations and conic inclusions \cite{ye1994nl,nesterov1999infeasible,o2016conic}:
\begin{equation}\label{opt: embedding}
\begin{aligned}
    &v=Su, \enskip u\in\overline{\mathbb{K}}, \enskip v\in -\overline{\mathbb{K}}^\circ,
\end{aligned}
\end{equation}
where 
\begin{equation}\label{eqn: S def}
    S=\begin{bmatrix}
    0 & H^\top & q\\
    -H & 0 & g\\
    -q^\top & -g^\top & 0
    \end{bmatrix}, \enskip \overline{\mathbb{K}}=\mathbb{R}^n\times (-\mathbb{K}^\circ)\times \mathbb{R}_+.
\end{equation}
Here \(-\mathbb{K}^\circ\) is also known as the \emph{dual cone} of \(\mathbb{K}\). Notice that the conditions in \eqref{opt: embedding} are satisfied if and only if
\(\xi=\begin{bmatrix}
u^\top & v^\top
\end{bmatrix}^\top\) is an optimal solution of optimization \eqref{opt: DRS} with
\begin{equation}\label{eqn: embedding lr}
    \begin{aligned}
        \ell(u, v)=&\delta_{\overline{\mathbb{K}}\times (- \overline{\mathbb{K}}^\circ)}(u, v), \enskip r(u, v)=\delta_{v=Su}(u, v),
    \end{aligned}
\end{equation}
where \(\delta_{v=Su}\) is the indicator function of set \(\{(u, v)|v=Su\}\).

If \(u, v\in\mathbb{R}^{m+n+1}\) satisfy \eqref{opt: embedding} without both being zero vectors, then one can detect the primal and dual infeasibility of optimization~\ref{opt: linear} as follows. Let 
\begin{equation*}
    \begin{aligned}
        &z=[u]_{1:n},\enskip w=[u]_{n+1:n+m},\\
        &\tau=[u]_{m+n+1},\enskip \kappa=[v]_{m+n+1}.
    \end{aligned}
\end{equation*}
If \(\tau>0\) and \(\kappa=0\), then \(\frac{1}{\tau}z\) satisfy the optimality conditions of optimization \eqref{opt: linear}. If \(\tau=0\) and \(\langle g, w\rangle<0\), then one can construct a proof of primal infeasibility. If \(\tau=0\) and \(\langle q, z\rangle<0\), then one can construct a proof of dual infeasibility. See \cite[Sec. 2.3]{o2016conic} for a detailed discussion. 

Algorithm~\ref{alg: embedding} summarizes the above infeasibility detection method, where the optimization \eqref{opt: DRS} defined by \eqref{eqn: embedding lr} is solved using DRS. In this case, the DRS iteration in \eqref{eqn: DRS iter} is equivalent to line \ref{eqn: DRS1 start}-\ref{eqn: DRS1 end} of Algorithm~\ref{alg: embedding} \cite[Sec. 3.2]{o2016conic}. Further, with a proper choice of initial values of \(u^1, v^1\) and a sufficiently large iteration number \(k\), the vectors \(u^k\) and \(v^k\) approximately satisfy \eqref{opt: embedding} without being both zero vectors \cite[Sec. 3.4]{o2016conic}. Notice that line~\ref{eqn: DRS1 start} computes the inverse of a symmetric matrix in \(\mathbb{R}^{(m+n+1)\times (m+n+1)}\).  

\begin{algorithm}[!ht]
\caption{DRS for optimization  \eqref{opt: linear} via homogeneous self-dual embedding}
\begin{algorithmic}[1]
\Require \(k, \epsilon\), initial values \(u^1, v^1\).
\Ensure \(z^\star\), ``Primal infeasible",``Dual Infeasible"
\For{\(j=1, 2, \ldots, k-1\)}
\State{\(\tilde{u}^{j+1}=(I+S)^{-1}(u^j+v^j)\)}\label{eqn: DRS1 start}
\State{\(u^{j+1}=\pi_{\overline{\mathbb{K}}}[\tilde{u}^{j+1}-v^j]\)}
\State{\(v^{j+1}=v^j-\tilde{u}^{j+1}+u^{j+1}\)}\label{eqn: DRS1 end}
\EndFor
\If{\([u^k]_{m+n+1}>\epsilon, [v^k]_{m+n+1}\leq \epsilon\)}\label{eqn: DRS1 test start}
\State {\Return{\(z^\star=\frac{1}{[u^k]_{m+n+1}}[u^k]_{1:n}\). }}
\ElsIf{\([u^k]_{m+n+1}\leq \epsilon\)}
\If{\(\langle g, [u^k]_{n+1:n+m}\rangle< -\epsilon\)}
\State {\Return{``Primal infeasible"}}
\EndIf
\If{\( \langle q, [u^k]_{1:n}\rangle< -\epsilon\)}
\State{\Return{``Dual infeasible"}}
\EndIf
\EndIf
\label{eqn: DRS1 test end}
\end{algorithmic}
\label{alg: embedding}
\end{algorithm}

\subsubsection{Minimal-displacement vector method}

Another way to detects the primal and dual infeasibility of optimization \eqref{opt: conic} via DRS is to compute the minimal-displacement vector of DRS. In particular, let
\[\overline{H} = \begin{bmatrix}
    H^\top &
    I
    \end{bmatrix}^\top,\enskip \overline{\mathbb{D}} = \{g+y|y\in\mathbb{K}\}\times \mathbb{D}.\]
Then one can rewrite optimization \eqref{opt: conic} by as an instance of optimization \eqref{opt: DRS} with \(\xi=\begin{bmatrix}
z^\top & y^\top
\end{bmatrix}^\top\) and
\begin{equation*}
    \begin{aligned}
        \ell(z, y)=&\frac{1}{2}z^\top P z+q^\top z+\delta_{\overline{H}z=y}(z, y),\enskip  r(z, y)= \delta_{\overline{\mathbb{D}}}(y),
    \end{aligned}
\end{equation*}
where \(\delta_{\overline{H}z=y}(z, y)\) is the indicator function of set \(\{z, y| \overline{H}z=y\}\). In addition, when DRS is applied to the above instance of optimization \eqref{opt: DRS}, the difference between the consecutive iterates converges to a limit, known as the minimal-displacement vector of DRS. If this limit is zero, the iterates of DRS asymptotically satisfies a set of primal-dual optimality conditions. If this limit is nonzero, the iterates of DRS asymptotically provide a proof for either primal or dual infeasibility \cite{banjac2019infeasibility,banjac2021asymptotic}.  

Algorithm~\ref{alg: splitting} summarizes the above infeasibility detection method, where the DRS iteration in \eqref{eqn: DRS iter} simplifies to line~\ref{eqn: DRS2 start}-\ref{eqn: DRS2 end} \cite[Sec. 4]{banjac2021asymptotic}. With a sufficient large iteration number \(k\), the vectors \(w^k-w^{k-1}\) and \(v^k-v^{k-1}\)  together give an approximation of the aforementioned minimal-displacement vector. Notice that line~\ref{eqn: M inverse} computes the inverse of a symmetric matrix in \(\mathbb{R}^{n\times n}\).  

\begin{algorithm}[!ht]
\caption{DRS for optimization \eqref{opt: conic} via minimal-displacement vector}
\begin{algorithmic}[1]
\Require \(k, \alpha, \epsilon\) initial values  \(z^1, y^1\).
\Ensure \(z^\star\), ``Primal infeasible", ``Dual infeasible".
\For{\(j=1, 2, \ldots, k-1\)}
\State{\(\tilde{y}^j=\pi_{\overline{\mathbb{D}}}[y^j]\)}\label{eqn: DRS2 start}
\State{\(w^j=y^j-\tilde{y}^j\)}
\State{\(\tilde{z}^j=(I+P+\overline{H}^\top \overline{H})^{-1}(z^j-q+\overline{H}^\top (2\tilde{y}^j-y^j))\)}\label{eqn: M inverse}
\State{\(z^{j+1}=z^j+\alpha(\tilde{z}^j-z^j)\)}
\State{\(y^{j+1}=y^j+\alpha(\overline{H}\tilde{z}^j-\tilde{y}^j)\)}\label{eqn: DRS2 end}
\EndFor
\If{\(\max\{\norm{w^k-w^{k-1}},\norm{z^k-z^{k-1}}\}\leq \epsilon\)}\label{eqn: DRS2 test start}
\State{\Return{\(z^\star=z^k\)}} 
\Else 
\If{\(\norm{w^k-w^{k-1}}> \epsilon\)}
\State{\Return{``Primal infeasible"}}
\EndIf
\If{\(\norm{z^k-z^{k-1}}> \epsilon\)}
\State{\Return{``Dual infeasible"}}
\EndIf
\EndIf\label{eqn: DRS2 test end}
\end{algorithmic}
\label{alg: splitting}
\end{algorithm}

\section{Proportional-integral projected gradient method}
\label{sec: PIPG}
We now introduce \emph{proportional-integral projected gradient method (PIPG)}, a primal-dual infeasibility detection method for optimization \eqref{opt: conic}. Provided that an optimal primal-dual solution exists, PIPG ensures that both the primal-dual gap and the constraint violation converge to zero along certain sequences of averaged iterates \cite{yu2020proportional,yu2021proportional}. In the following, we will show that the iterates of PIPG can also provide proof of primal or dual infeasibility.

Algorithm~\ref{alg: PIPG} summarizes PIPG for infeasibility detection, where \(k\in\mathbb{N}\) is the maximum number of iteration, \(\alpha>0\) is the step size, and \(\epsilon\in\mathbb{R}_+\) is chosen such that numbers in interval \([-\epsilon, \epsilon]\) are treated as approximately zero. Unlike Algorithm~\ref{alg: embedding} and Algorithm~\ref{alg: splitting}, Algorithm~\ref{alg: PIPG} does not compute matrix inverse. Furthermore, using Lemma~\ref{lem: Moreau}, one can show that the projections in Algorithm~\ref{alg: PIPG} are no more difficult to compute than those in Algorithm~\ref{alg: splitting}.

\begin{algorithm}[!ht]
\caption{PIPG for optimization \eqref{opt: conic} }
\begin{algorithmic}[1]
\Require \(k, \alpha, \epsilon\), initial values  \(z^1, v^1\).
\Ensure \(z^\star\), ``Primal infeasible", ``Dual infeasible".
\For{\(j=1, 2, \ldots, k-1\)}\label{alg: pipg for start}
\State{\(w^{j+1}= \pi_{\mathbb{K}^\circ}[v^j+\alpha(Hz^j-g)]\)}\label{alg: w}
\State{\(z^{j+1}=\pi_{\mathbb{D}}[z^j-\alpha( Pz^j+q+H^\top w^{j+1})]\)}\label{alg: z}
\State{\(v^{j+1}=w^{j+1}+\alpha H(z^{j+1}-z^j)\)}\label{alg: v}
\EndFor\label{alg: pipg for end}
\If{\(\max\{\norm{w^k-w^{k-1}},\norm{z^k-z^{k-1}}\}\leq \epsilon\)}\label{eqn: PIPG test start}
\State{\Return{\(z^\star=z^k\)}} 
\Else 
\If{\(\norm{w^k-w^{k-1}}> \epsilon\)}
\State{\Return{``Primal infeasible"}}
\EndIf
\If{\(\norm{z^k-z^{k-1}}> \epsilon\)}
\State{\Return{``Dual infeasible"}}
\EndIf
\EndIf\label{eqn: PIPG test end}
\end{algorithmic}
\label{alg: PIPG}
\end{algorithm}

The name proportional-integral projected gradient method (PIPG) is due to the fact that line~\ref{alg: w}-\ref{alg: v} of Algorithm~\ref{alg: PIPG} combine projected gradient method together with proportional-integral feedback of constraints violation \cite{yu2021proportional}. Similar ideas were first introduced in distributed optimization \cite{wang2010control,yu2020mass,yu2020rlc} and later extended to general conic optimization \cite{yu2020proportional,yu2021proportional}.

In the following, we will show that line~\ref{alg: w}-\ref{alg: v} in Algorithm~\ref{alg: PIPG} is the fixed-point iteration of an averaged operator. Further, the minimal displacement vector (see Lemma~\ref{lem: min dis}) of said averaged operator yield information regarding optimality and infeasibility. We will use the following assumptions on optimization \eqref{opt: conic} and the step size \(\alpha\) in Algorithm~\ref{alg: PIPG}.

\begin{assumption}\label{asp: opt}
\begin{enumerate}
    \item Matrix \(P\in\mathbb{R}^{n\times n}\) is symmetric and positive semidefinite.
    \item Set \(\mathbb{D}\subseteq\mathbb{R}^n\) and cone \(\mathbb{K}\subseteq\mathbb{R}^m\) are both nonempty, closed and convex. 
\end{enumerate}
\end{assumption}

\begin{assumption}\label{asp: step size}
The step size \(\alpha\) in Algorithm~\ref{alg: PIPG} satisfies
\begin{equation*}
    0<\alpha\leq \frac{8-\frac{4}{\gamma}}{\sqrt{\lambda^2+16\nu^2}+\lambda}
\end{equation*}
for some \(\gamma\in(\frac{1}{2}, 1)\), \(\lambda\geq \mnorm{P}\), and \(\nu\geq \mnorm{H}\). 
\end{assumption}

As our first step, the following proposition shows that, under Assumption~\ref{asp: opt} and Assumption~\ref{asp: step size}, line~\ref{alg: w}-\ref{alg: v} in Algorithm~\ref{alg: PIPG} is the fixed-point iteration of an averaged operator. 

\begin{proposition}\label{prop: PIPG avg op}
Suppose that Assumption~\ref{asp: opt} and Assumption~\ref{asp: step size} hold. Then line~\ref{alg: w}-\ref{alg: v} in Algorithm~\ref{alg: PIPG} implies that
\begin{equation}\label{eqn: pipg operator}
    \begin{bmatrix}
    z^{j+1}\\
    v^{j+1}
    \end{bmatrix}=T\left(\begin{bmatrix}
    z^j\\
    v^j
    \end{bmatrix}\right),
\end{equation}
where \(T:\mathbb{R}^{m+n}\to \mathbb{R}^{m+n}\) is a \(\gamma\)-averaged operator, where \(\gamma\in(\frac{1}{2}, 1)\) is given in Assumption~\ref{asp: step size}.
\end{proposition}

\begin{proof}
See Appendix~\ref{app: lemA}.
\end{proof}

\begin{remark}
In order to satisfy Assumption~\ref{asp: step size}, one needs to estimate the values of \(\mnorm{H}\) and \(\mnorm{P}\). The following \emph{power iteration} ensures that \(\norm{z^j}\) quickly converges to \(\mnorm{H}^2\) as the number of iteration \(j\) increases:
\begin{equation*}
    z^{j+1}=\frac{1}{\norm{z^j}} H^\top H z^j,
\end{equation*}
where the entries in vector \(z^1\in\mathbb{R}\) are randomly generated. Notice that the above power iteration does not compute matrix inverse or decomposition; see \cite{kuczynski1992estimating} for further details. The computation of \(\mnorm{P}\) is similar.
\end{remark}

Due to Proposition~\ref{prop: PIPG avg op} and Lemma~\ref{lem: min dis}, one can show that both \(\norm{z^k-z^{k-1}}\) and \(\norm{w^k-w^{k-1}}\) computed in Algorithm~\ref{alg: PIPG} converge to certain limits as \(k\) increases. To show that these limits yield information regarding optimality and infeasibility, we first introduce the following results.

\begin{lemma}\label{lem: dist func}
Suppose that Assumption~\ref{asp: opt} holds and the sequence \(\{w^j, z^j, v^j\}_{j\in\mathbb{N}}\) is computed recursively using line~\ref{alg: w}-\ref{alg: v} in Algorithm~\ref{alg: PIPG} where \(\alpha>0\). Let \(\lambda\geq \mnorm{P}\) and \(\nu\geq \mnorm{H}\). Then, for all \(j\geq 3\),
\begin{equation}
\begin{aligned}
    &d(-Pz^j-q-H^\top w^j|N_{\mathbb{D}}(z^j))\leq \left(\frac{1}{\alpha}+\lambda\right)\norm{z^j-z^{j-1}},\\
    & d(Hz^j-g|N_{\mathbb{K}^\circ}(w^j))\\
    &\leq \frac{1}{\alpha}\norm{w^j-w^{j-1}}+\nu\norm{z^j-2z^{j-1}+z^{j-2}}.
\end{aligned}
\end{equation}  
\end{lemma}

\begin{proof}
See Appendix~\ref{app: dist func}.
\end{proof}

\begin{lemma}\label{lem: separating}
Suppose that Assumption~\ref{asp: opt} and Assumption~\ref{asp: step size} hold and the sequence \(\{w^j, z^j, v^j\}_{j\in\mathbb{N}}\) is computed recursively using line~\ref{alg: w}-\ref{alg: v} in Algorithm~\ref{alg: PIPG}. Then there exists \(\overline{z}\in\rec \mathbb{D}\) and \(\overline{w}\in\mathbb{K}^\circ\) such that \(\lim\limits_{j\to\infty} z^j-z^{j-1}=\overline{z}\) and \(\lim\limits_{j\to\infty}w^j-w^{j-1}=\overline{w}\).
Further,
\begin{equation}\label{eqn: pd inf}
\begin{aligned}
    &H\overline{z}\in\mathbb{K}, \enskip P\overline{z}=0, \enskip \langle q, \overline{z}\rangle=-\frac{1}{\alpha}\norm{\overline{z}}^2,\\ &\sigma_{\mathbb{D}}(-H^\top \overline{w})+\langle g, \overline{w}\rangle= -\frac{1}{\alpha}\norm{\overline{w}}^2.
\end{aligned}
\end{equation}
\end{lemma}
\begin{proof}
See Appendix~\ref{app: lemB}.
\end{proof}

Equipped with Lemma~\ref{lem: dist func} and Lemma~\ref{lem: separating}, we summarize our main theoretical contribution in the following theorem.
 
\begin{theorem}\label{thm: infeasibility}
Suppose that Assumption~\ref{asp: opt} and Assumption~\ref{asp: step size} hold and the sequence \(\{w^j, z^j, v^j\}_{j\in\mathbb{N}}\) is computed recursively using line~\ref{alg: w}-\ref{alg: v} in Algorithm~\ref{alg: PIPG}. Then there exist \(\overline{z}\in\rec\mathbb{D}\) and \(\overline{w}\in\mathbb{K}\) such that
\begin{equation}\label{eqn: limit points}
    \lim_{j\to\infty}z^j-z^{j-1}=\overline{z}, \enskip \lim_{j\to\infty}w^j-w^{j-1}=\overline{w}.
\end{equation}
Further, the following three statements hold.
\begin{enumerate}
    \item If \(\norm{\overline{z}}=\norm{\overline{w}}=0\), then
\begin{equation}
\begin{aligned}
    &\lim_{j\to\infty} d(-Pz^j-q-H^\top w^j|N_{\mathbb{D}}(z^j))=0,\\
    &\lim_{j\to\infty} d(Hz^j-g| N_{\mathbb{K}^\circ}(w^j))=0.
\end{aligned}\label{eqn: optimality limit}
\end{equation}    
    \item If \(\norm{\overline{w}}> 0\), then 
    \begin{equation}\label{eqn: separating}
     \inf_{z\in\mathbb{D}}\,\langle Hz-g, \overline{w}\rangle>0=\sup_{y\in\mathbb{K}}\,\langle y, \overline{w}\rangle .  
    \end{equation}
    \item If \(\norm{\overline{z}}> 0\), then
    \begin{equation}\label{eqn: improving}
    H\overline{z}\in\mathbb{K}, \enskip P\overline{z}=0, \enskip \langle q, \overline
    z\rangle<0.
    \end{equation}
\end{enumerate}
\end{theorem}

\begin{proof} Lemma~\ref{lem: min dis} implies that the limits in \eqref{eqn: limit points} hold for some \(\overline{z}\in\rec\mathbb{D}\) and \(\overline{w}\in\mathbb{K}^\circ\). Further, the first statement is due to Lemma~\ref{lem: dist func} and \eqref{eqn: limit points}. The second and the third statement are due to \eqref{eqn: pd inf} and the fact that \(\sigma_{\mathbb{D}}(-H^\top \overline{w})=-\inf_{z\in\mathbb{D}}\,\langle H z, \overline{w}\rangle\) and \(\sup_{y\in\mathbb{K}}\,\langle y, \overline{w}\rangle=0\) when \(\overline{w}\in\mathbb{K}^\circ\). The latter fact is due to the definition of polar cone in \eqref{eqn: polar cone}.
\end{proof}

We now discuss the implications of the three statements in Theorem~\ref{thm: infeasibility}. First, the two limits in \eqref{eqn: optimality limit} imply that a set of primal-dual optimality conditions are satisfied asymptotically. To see this implication, let \(z^\star\in\mathbb{D}\) and \(w^\star\in\mathbb{K}^\circ\) be such that
\[d(-Pz^\star-q-H^\top w^\star|N_{\mathbb{D}}(z^\star))= d(Hz^\star-g| N_{\mathbb{K}^\circ}(w^\star))=0.\]
Since the point-to-set distance is nonnegative, the above conditions are equivalent to the following primal-dual optimality conditions of optimization~\eqref{opt: conic} \cite[Ex.11.52]{rockafellar2009variational}:
\begin{subequations}\label{eqn: optimality}
\begin{align}
     -Pz^\star-q-H^\top w^\star\in & N_{\mathbb{D}}(z^\star),\\
    Hz^\star-g\in & N_{\mathbb{K}^\circ}(w^\star).\label{eqn: dual normal}
\end{align}
\end{subequations}
Provided that the constraints in optimization~\eqref{opt: conic} satisfy certain qualification, the conditions in \eqref{eqn: optimality} imply that \((z^\star, w^\star)\) is an optimal primal-dual solution of optimization \eqref{opt: conic} \cite[Cor. 11.51]{rockafellar2009variational}. We note that even if the conditions in \eqref{eqn: optimality limit} hold, optimization \eqref{opt: conic} can still be infeasible, \ie, its constraints cannot be satisfied. In this case, however, an infinitesimal perturbation of vector \(g\) will make optimization \eqref{opt: conic} either feasible or infeasible. See \cite[Sec. 6.3]{banjac2019infeasibility} for an example.

Second, the strict inequality in \eqref{eqn: separating} implies that \(\langle\overline{w}, y\rangle=0\) is a hyperplane that separates cone \(\mathbb{K}\) from the set \(\{Hz-g|z\in\mathbb{D}\}\). See Fig.~\ref{fig: seperate} for an illustration. As a result,  optimization \eqref{opt: conic} is infeasible.

\begin{figure}
\centering
    \begin{tikzpicture}[scale=0.35]
    
    \coordinate (O) at (0, 0) ;
    \coordinate (K1) at (30:4);
    \coordinate (K2) at (-30:4);
    
    \filldraw[draw=blue!10, fill=blue!15] (O) -- (K1) -- (K2) -- cycle;
    
   \draw[-] (O) -- (K1);
   \draw[-] (O) -- (K2);
 
   \draw [black,fill=blue!15] (-5, 0) ellipse (3.5 and 2.5);

   \node[label=right:{\scriptsize $\mathbb{K}$}] at (1.2, 0) {};
   \node at (-5, 0) {\scriptsize$\{Hz-g|z\in\mathbb{D}\}$};
   
   \draw[thick, red] (-120:4) -- (60:4);
   \node[label=right:{\scriptsize $\langle \overline{w}, y\rangle=0$}] at (-120:3.5) {};
   
   \fill (O) circle [radius=2pt];
   \end{tikzpicture}
   \caption{An illustration of the seperating hyperplane.}
   \label{fig: seperate}
\end{figure}
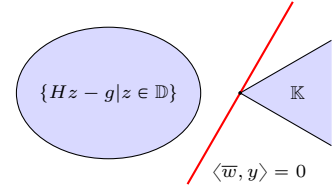

Third, if there exists \(z\in\mathbb{D}\) such that \(Hz-g\in\mathbb{K}\), then the conditions in \eqref{eqn: improving} imply that the optimal value of optimization \eqref{opt: conic} is unbounded. In particular, let \(y=z+\overline
z\), then \eqref{eqn: improving} and the fact that \(\overline{z}\in\rec \mathbb{D}\) and \(H\overline{z}\in\mathbb{K}\) imply the following:
\[y\in\mathbb{D},\enskip  Hy-g\in\mathbb{K}, \enskip \frac{1}{2}y^\top Py+q^\top y<\frac{1}{2}z^\top Pz+q^\top z.\]
Hence the value of the objective function in optimization \eqref{opt: conic} strictly decreases along direction \(\overline{z}\) without violating the constraints. By repeating a similar process we can show the the optimal value of optimization \eqref{opt: conic} is indeed unbounded from below. Furthermore, the dual problem of optimization \eqref{opt: conic} is infeasible \cite{banjac2019infeasibility,banjac2021asymptotic}.

\section{Applications in constrained optimal control}
\label{sec: experiment}

We demonstrate the application of PIPG in constrained optimal control problems. In particular, we will show how to use PIPG to compute minimum-time landing trajectory in Section~\ref{subsec: quasi cvx}, and reduce the number of binary variables in nonconvex path-planning problem in Section~\ref{subsec: mixed}.

\begin{figure}[!ht]
    \centering
    \includegraphics[width=0.5\linewidth]{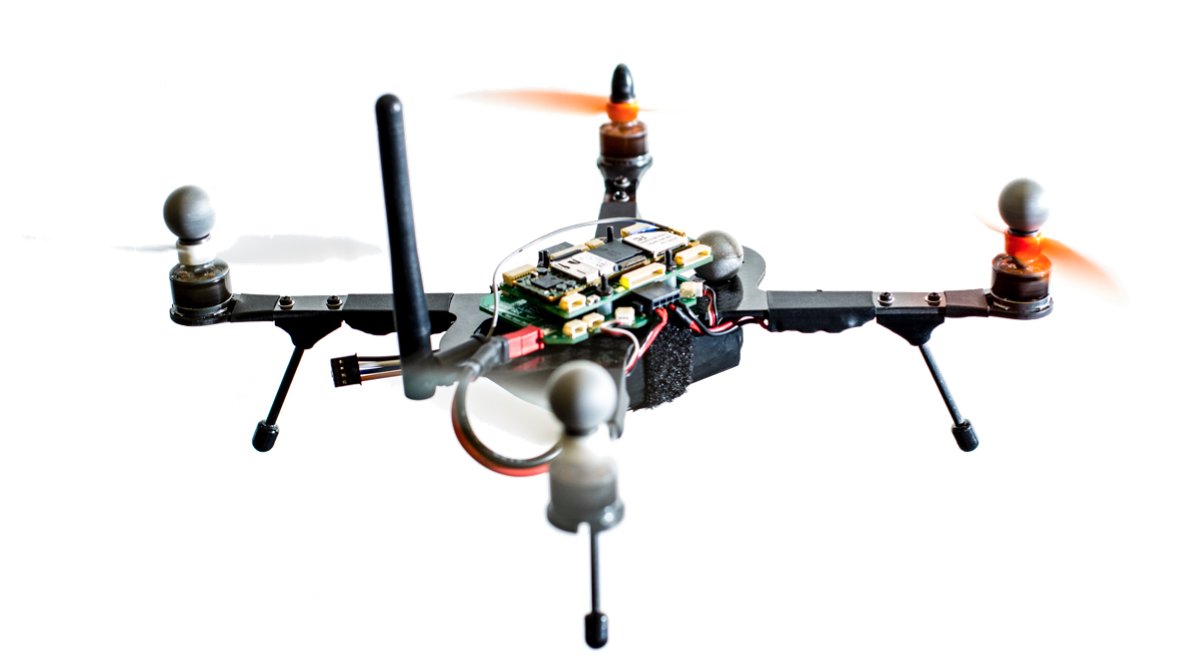}
    \caption{A custom-made quadrotor designed by the Autonomous Control Laboratory.}
    \label{fig: miki quad}
\end{figure}

We consider the optimal control of the custom-made quadrotor designed by the Autonomous Control Laboratory (ACL) at University of Washington, which we referred to as the \emph{ACL quadrotor}. See Fig.~\ref{fig: miki quad} for an illustration and \cite{szmuk2019successive} for a detailed description of the ACL quadrotor. For simplicity, we let all problem parameters in the following (\eg, mass of the quadrotor, maximum thrust magnitude) be unitless.

The dynamics of an ACL quadrotor is given by 
\begin{equation}\label{sys: CT LTI}
    \frac{d}{ds}x(s)=A_cx(s)+B_cu(s)+h_c
\end{equation}
for all time \(s\in\mathbb{R}_+\), where \(x(s)\in\mathbb{R}^{6}\) denotes the position and velocity of the center of mass of the quadrotor at time \(s\), \(u(s)\in\mathbb{R}^{3}\) denotes the thrust vector provided by the propellers of the quadrotor at time \(s\). In addition, 
\begin{equation*}
    A_c = \begin{bmatrix}
    0_{3\times 3} &  I_3\\
    0_{3\times 3} & 0_{3\times 3}
    \end{bmatrix},\, B_c=\frac{1}{0.35}\begin{bmatrix}
    0_{3\times 3}\\
     I_3
    \end{bmatrix}, \, h_c=\begin{bmatrix}
    0_{5\times 1} \\ -9.8
    \end{bmatrix},
\end{equation*}
where \(0.35\) is the total mass of the quadrotor, \(9.8\) is the gravity constant.

We can further simplify dynamics \eqref{sys: CT LTI} by assuming the thrust vector does not change value within each sampling time period of length \(\Delta\in\mathbb{R}_+\). In particular, suppose that
\begin{equation}\label{eqn: zoh}
    u(s)=u(t\Delta), \enskip t\Delta\leq s<(t+1)\Delta
\end{equation}
for all \(t\in\mathbb{N}\). Let
\(x_t\coloneqq x(t\Delta)\) and \(u_t\coloneqq u(t\Delta)\) for all \(t\in\mathbb{N}\). Then dynamics equation \eqref{sys: CT LTI} is equivalent to
\begin{equation}\label{sys: DT LTI}
    x_{t+1}=Ax_t+Bu_t+h
\end{equation}
for all \(t\in\mathbb{N}\), where
\begin{equation}
\begin{aligned}
     &\textstyle A=\exp(A_c\Delta), \enskip B=\left(\int_0^\Delta \exp(A_c s)ds\right)B_c,\\
     & \textstyle h=\left(\int_0^\Delta \exp(A_c s)ds\right)h_c.
\end{aligned}
\end{equation}
Due to the physical limits of its onboard motors, we consider the following constraints on the thrust vector of the quadrotor. The magnitude of the thrust vector is lower bounded by \(\rho_1=2\) and upper bounded by \(\rho_2=5\). Further, the direction of the thrust vector is confined within an icecream cone with half-angle angle \(\theta=\pi/4\). These constraints can be written as \(u_t\in\mathbb{U}\) for all \(t\in\mathbb{N}\), where
\begin{equation}\label{eqn: muffin}
    \mathbb{U}\coloneqq\{u\in\mathbb{R}^{3}|\norm{u}\cos\theta\leq [u]_3,[u]_3\geq \rho_1,\norm{u}\leq \rho_2\}.
\end{equation}
Notice that we approximate the lower bound on \(\norm{u}\) using a linear inequality constraint \([u]_3\geq 2\) so that set \(\mathbb{U}\) remains convex. See Fig.~\ref{fig: muffin} for an illustration of set \(\mathbb{U}\). We also note that the state of system \eqref{sys: DT LTI} can be transferred from any initial value to any final value using a finite-length sequence of inputs, where each input is chosen from set \(\mathbb{U}\). 

\begin{figure}
    \centering
    \begin{tikzpicture}[scale=0.4]
\draw[color=blue, fill=blue!15] (135:2.82) --  (135:5) arc (135:45:5) -- (45:2.82) -- cycle;

\draw[-, color=blue] (135:5) -- (45:5);

\draw[dashed] (135:2.82) -- (0,0) (45:2.82) -- (0, 0) (0, 0) --(2, 0);

\draw[-] (2, 2) --(2.8, 2) (2, 0)--(2.8, 0);
\draw[latex-latex] (2.4, 2) -- (2.4, 0) node[midway,right]{\footnotesize $\rho_1$};

\draw[-]  (135:0.8) arc(135:45:0.8) node[midway,above]{\footnotesize $2\theta_1$};
\draw[rotate around={45:(0,0)}] (-0.8, 5) -- (0, 5) (-0.8, 0) -- (0, 0);
\draw[latex-latex, rotate around={45:(0,0)}] (-0.4, 5) -- (-0.4, 0) node[midway,left]{\footnotesize $\rho_2$};


\end{tikzpicture}
    \caption{An illustration of the set of feasible thrust vectors.}
    \label{fig: muffin}
\end{figure}
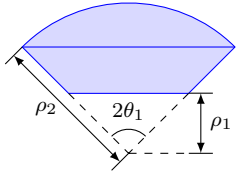

In the following, we will consider optimal control problems with dynamics constraints \eqref{sys: DT LTI} (where we fix \(\Delta=0.2\)) and input constraint set \eqref{eqn: muffin}. We will solve these optimal control problems using Algorithm~\ref{alg: splitting} (where \(\alpha=1\)) and Algorithm~\ref{alg: PIPG} (where \(\alpha\) satisfies Assumption~\ref{asp: step size} with \(\gamma=0.9\)) and compare their performance. We will also use the following notation:
\begin{equation}\label{eqn: traj}
\begin{aligned}
 &u_{[0, \tau-1]}\coloneqq \begin{bmatrix}
 u_0^\top & u_1^\top & \ldots & u_{\tau-1}^\top 
 \end{bmatrix}^\top,\\
& x_{[1, \tau-1]}\coloneqq\begin{bmatrix}
x_1^\top & x_2^\top & \ldots &  x_{\tau-1}^\top
\end{bmatrix}^\top.  
\end{aligned}
\end{equation}

\subsection{Minimum-time landing via quasiconvex optimization}
\label{subsec: quasi cvx}

We consider the problem of landing a quadrotor on a charging station in the minimum amount of time possible subject to various state and input constraints. Let \(x_0\in\mathbb{R}^6\) denote the initial state of the quadrotor, and suppose that the charging station is located at the origin of the position coordinate system. Then the minimum landing time is given by \(i\Delta\), where \(\Delta\) is the sampling time in \eqref{eqn: zoh}, and \(i\) is the smallest integer that makes the following optimization feasible (\ie, there exists one solution that satisfies its constraints):
\begin{equation}\label{opt: landing}
    \begin{array}{ll}
        \underset{u_{[0, \tau-1]}, x_{[1, \tau-1]}}{\mbox{minimize}}  & \frac{1}{2}\sum_{t=0}^{\tau-1} \norm{u_t}^2\\
        \quad\mbox{subject to}  &  x_{t+1}=Ax_t+Bu_t+h,\,0\leq t\leq \tau-1,\\
        &u_t\in\mathbb{U},\, 0\leq t\leq \tau-1,\\
        & x_{t}\in\mathbb{X},\,  1\leq t\leq \tau-1,\\
        & x_t=\mathbf{0}, \, i\leq t\leq \tau-1, 
    \end{array}
\end{equation}
where \(\tau\in\mathbb{N}\) is large enough such that optimization \eqref{opt: landing} is feasible if \(i=\tau-1\), set \(\mathbb{U}\) is given by \eqref{eqn: muffin}, and set \(\mathbb{X}\) is given by
\begin{equation*}
         \mathbb{X} =\{r\in\mathbb{R}^3| \norm{r}\cos\beta\leq [r]_3\} \times \{r\in\mathbb{R}^3|\norm{r}\leq \eta\},
\end{equation*}
where \(\beta=\pi/4\), \(\eta=5\). The quadratic objective function in \eqref{opt: landing} penalizes large-magnitude inputs. The constraint \(x_t\in\mathbb{X}\) upper bounds the speed of the quadrotor and limits the directions from which the quadrotor approach the charging station. The latter is also known as the \emph{approach cone} constraint, which is widely used in spacecraft landing \cite{malyuta2021advances}. See Fig.~\ref{fig: landing} for an illustration. 

\begin{figure}[!ht]
	\centering
	\begin{tikzpicture}[scale=0.7]
		\coordinate (O) at (0, 0);
		\coordinate (quad) at (-1.8, 4);
		\coordinate (mid1) at (-1.7, 2.8);
		\coordinate (mid2) at (0.6, 1.5);
		
		\fill[blue!10] (0,0) -- (55:5.5) -- (125:5.5) -- cycle;
		
		\node [label={[align=center, blue]\scriptsize approach\\ \scriptsize cone}] at (3, 2) {};
		
		
		\node [quadcopter side,fill=white,draw=black,minimum width=1cm,rotate=-15,scale=1] at (quad) {}; 
		
		\draw[dashed,-latex] (quad) to[out=-85,in=105] (mid1); 
		\draw[dashed,-latex] (mid1) to[out=-85,in=170] (mid2);
		\draw[dashed] (mid2) to[out=-60,in=90] (O);
		
		\fill (O) circle [radius=1.5pt];

		\draw (-1, 0) rectangle +(2, -0.2);
		\pattern[pattern=crosshatch] (-1, 0) rectangle +(2, -0.2);
		\node[label=below:{\scriptsize charging station }] at (0, -0.2) {};
	\end{tikzpicture}
		\caption{An illustration of the quadrotor landing.}
		\label{fig: landing}
\end{figure}
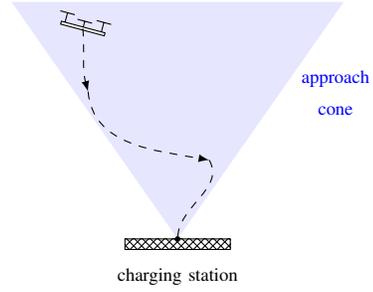

One can compute the smallest integer \(i\) that makes optimization \eqref{opt: landing} feasible by solving a quasiconvex optimization problem. To this end, we say \((u_{[0, \tau-1]}, x_{[1, \tau-1]})\) is feasible for \eqref{opt: landing} if \((u_{[0, \tau-1]}, x_{[1, \tau-1]})\) satisfy the constraints in optimization \eqref{opt: landing}. We also define function \(f\) using \eqref{eqn: quasiconvex f} (see the bottom of the next page). 
\begin{figure*}[b]
\rule[1ex]{\textwidth}{0.1pt}
\begin{equation}\label{eqn: quasiconvex f}
f(u_{[0, \tau-1]}, x_{[1, \tau-1]})\coloneqq \min\left\{i\in\mathbb{N}\cup\{\infty\}\left|\text{\((u_{[0, \tau-1]}, x_{[1, \tau-1]})\) is feasible for \eqref{opt: landing}}
\right.\right\}, \enskip \text{where } \min \emptyset \coloneqq \infty.
\end{equation}
\end{figure*}
Then the smallest integer \(i\) that makes optimization \eqref{opt: landing} feasible is the optimal value of the following optimization:
\begin{equation}\label{opt: quasiconvex}
    \begin{array}{ll}
        \underset{u_{[0, \tau-1]}, x_{[1, \tau-1]}}{\mbox{minimize}} & f(u_{[0, \tau-1]}, x_{[1, \tau-1]}).
    \end{array}
\end{equation}

We will show that optimization \eqref{opt: quasiconvex} is a quasiconvex optimization problem and can be solved using bisection method and infeasibility detection as follows. First, given any \(\epsilon\in\mathbb{R}_+\), one can verify that 
\begin{equation}\label{eqn: sublevel}
    \begin{aligned}
     &\{(u_{[0, \tau-1]}, x_{[1, \tau-1]})|f(u_{[0, \tau-1]}, x_{[1, \tau-1]})\leq \epsilon\}\\
     &=\left\{(u_{[0, \tau-1]},
      x_{[1, \tau-1]})
      \left|\begin{aligned}
 &\text{\((u_{[0, \tau-1]}, x_{[1, \tau-1]})\) is}\\
 & \text{feasible for \eqref{opt: landing} with \(i=\lfloor \epsilon\rfloor\)}
\end{aligned}\right.\right\},
\end{aligned}
\end{equation}
where \(\lfloor \epsilon\rfloor\) is the largest integer less or equals to \(\epsilon\). Since the constraints in \eqref{opt: landing} are all convex with respect to \(u_{[0, \tau-1]}\) and \(x_{[1, \tau-1]}\), we conclude that set \eqref{eqn: sublevel} is convex. In other words, all sublevel sets of function \(f\) are convex sets. Therefore, function \(f\) is quasiconvex and one can solve optimization \eqref{opt: quasiconvex} by checking whether the set in \eqref{eqn: sublevel} is empty, \ie, whether optimization \eqref{opt: landing} is feasible when \(i=\lfloor\epsilon\rfloor\), for different choices of \(\epsilon\in\mathbb{R}_{+}\) \cite[Sec. 4.2.5]{boyd2004convex}. 

In the following, we consider optimization \eqref{opt: quasiconvex} where \(\tau=40\), \(i=\{24, 26\}\), and \(x_0=\begin{bmatrix}
6 & 6 & 15 & 2 & 2 & 2
\end{bmatrix}^\top\).

To solve the above optimization, we apply Algorithm~\ref{alg: splitting} and Algorithm~\ref{alg: PIPG} to the corresponding optimization \eqref{opt: landing}. In particular, Appendix~\ref{app: transform} gives the transformation from optimization~\eqref{opt: landing} to optimization \eqref{opt: conic}. Fig.~\ref{fig: landing exp} illustrates the convergence of \(\norm{w^j-w^{j-1}}\) and \(\norm{z^j-z^{j-1}}\) computed in Algorithm~\ref{alg: splitting} and Algorithm~\ref{alg: PIPG}, where the two algorithms show similar convergence. These convergence results show that both \(\norm{w^j-w^{j-1}}\) and \(\norm{z^j-z^{j-1}}\) converge to zero if \(i=26\), and \(\norm{w^j-w^{j-1}}\) does not converge to zero if \(i=24\). As a result of Theorem~\ref{thm: infeasibility}, we conclude that the optimal value of optimization \eqref{opt: quasiconvex} is between \(24\) and \(26\) (solving another instance of optimization \eqref{opt: landing} will confirm that it is \(25\)).

\begin{figure}[!ht]
    \centering
    \begin{subfigure}[b]{0.49\columnwidth}
    \centering 
    \includegraphics[width=\textwidth]{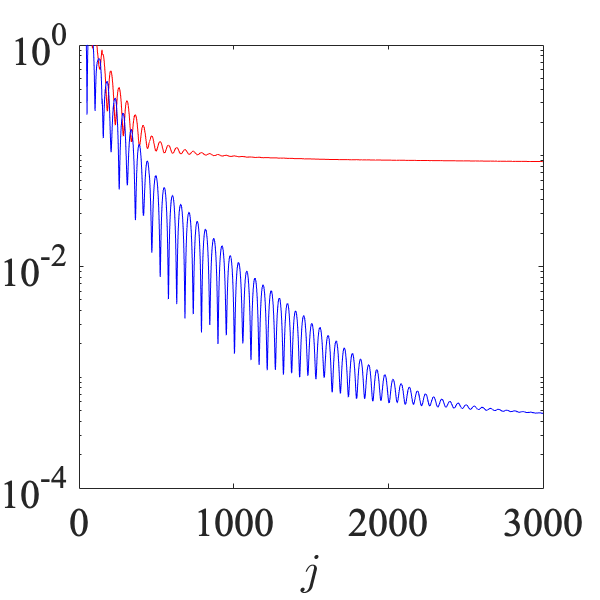}
    \caption{\(i=24\), Algorithm~\ref{alg: splitting}}
    \end{subfigure}
    \begin{subfigure}[b]{0.49\columnwidth}
    \centering 
    \includegraphics[width=\textwidth]{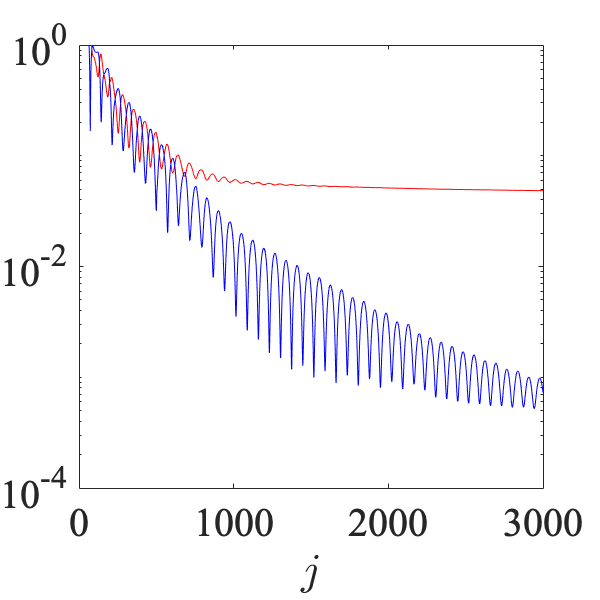}
    \caption{\(i=24\), Algorithm~\ref{alg: PIPG}}
    \end{subfigure}
    \centering
    \begin{subfigure}[b]{0.49\columnwidth}
    \centering 
    \includegraphics[width=\textwidth]{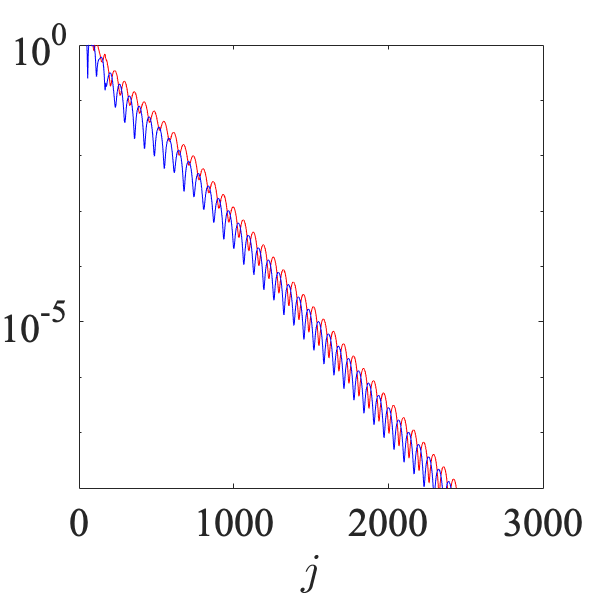}
    \caption{\(i=26\), Algorithm~\ref{alg: splitting}}
    \end{subfigure}
    \centering
    \begin{subfigure}[b]{0.49\columnwidth}
    \centering 
    \includegraphics[width=\textwidth]{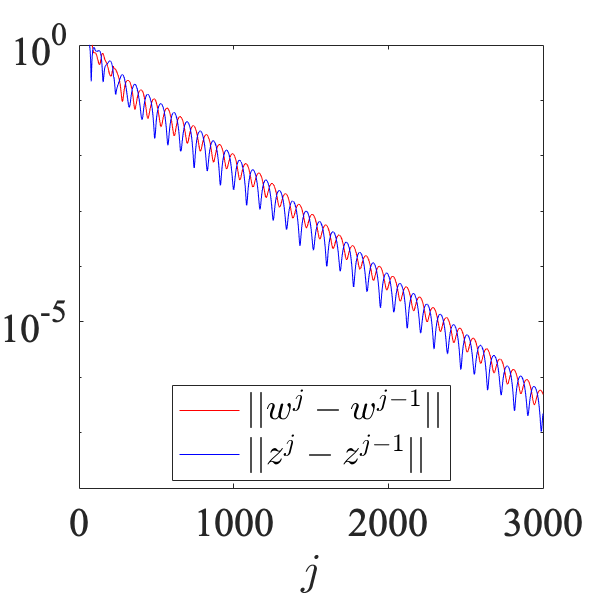}
    \caption{\(i=26\), Algorithm~\ref{alg: PIPG}}
    \end{subfigure}
    \caption{Convergence of \(\norm{w^j-w^{j-1}}\) and \(\norm{z^j-z^{j-1}}\) in Algorithm~\ref{alg: splitting} and Algorithm~\ref{alg: PIPG} when applied to optimization \eqref{opt: landing} with different values of \(i\). } 
    \label{fig: landing exp}
\end{figure}

We note that the quasiconvex optimization approach presented in this section has also been used for the minimum time control of overhead cranes \cite{zhang2014minimum}.

\subsection{Path-planning via mixed-integer optimization}
\label{subsec: mixed}

We consider the problem of flying a quadrotor from a initial position to a final position, both located in a L-shaped corridor. See Fig.~\ref{fig: corridor} for an illustration. To avoid collision with the boundary of the corridor, we consider the following state constraints
\begin{equation}
    \underline{r}^1\leq Tx_t\leq \overline{r}^1\text{ or }\underline{r}^2\leq Tx_t\leq \overline{r}^2,
\end{equation}
for all \(t\), where \(T=\begin{bmatrix}
I_3 & 0_{3\times 3}
\end{bmatrix}\), 
\begin{equation*}
\begin{aligned}
 &\underline{r}^1=\begin{bmatrix}
    0 & -2 & 0
    \end{bmatrix}^\top,\enskip \overline{r}^1=\begin{bmatrix}
    2 & 9 & 3
    \end{bmatrix}^\top,\\
    &\underline{r}^2=\begin{bmatrix}
    2 & -2 & 0
    \end{bmatrix}^\top,\enskip \overline{r}^1=\begin{bmatrix}
    12 & 0 & 3
    \end{bmatrix}^\top.
\end{aligned}
\end{equation*}
Given the initial and final state of the quadrotor (denoted by \(x_0\in\mathbb{R}^6\) and \(x_\tau\in\mathbb{R}^6\), respectively), the above problem can be formulated as follows:
\begin{equation}\label{opt: mixed}
    \begin{array}{ll}
        \underset{\substack{u_{[0, \tau-1]}, x_{[1, \tau-1]}\\ b_{[1, \tau-1]}} }{\mbox{minimize}}  & \frac{1}{2}\sum_{t=0}^{\tau-1} \norm{u_t}^2\\
        \quad\mbox{subject to} &  x_{t+1}=Ax_t+Bu_t+h,\, 0\leq t\leq \tau-1,\\
        & u_t\in\mathbb{U},\enskip 0\leq t\leq \tau-1,\\
        & x_{t}\in\mathbb{X},\enskip 1\leq t\leq \tau-1,\\
        & \underline{r}^1+b_t(\underline{r}^2-\underline{r}^1)\leq Tx_t,1\leq t\leq \tau-1,\\
        & \overline{r}^1+b_t(\overline{r}^2-\overline{r}^1)\geq Tx_t,1\leq t\leq \tau-1,\\
        & b_t\in\{0, 1\}, \, 1\leq t\leq \tau-1,
    \end{array}
\end{equation}
where the notation \(b_{[1, \tau-1]}\) is similar to those in \eqref{eqn: traj}, \(\mathbb{U}\) is given by \eqref{eqn: muffin}, and \(\mathbb{X} =\mathbb{R}^3\times \{r\in\mathbb{R}^3|\norm{r}\leq \eta\}\) with \(\eta=5\). The quadratic objective function in \eqref{opt: mixed} penalizes large-magnitude inputs. The constraint \(x_t\in\mathbb{X}\) upper bounds the speed of the quadrotor. We assume \(\tau\) is large enough such that optimization \eqref{opt: mixed} has an optimal solution.

\begin{figure}[!ht]
	\centering
	\begin{tikzpicture}[scale=0.6]
		
		\coordinate (O1) at (0, 0);
		\coordinate (O2) at (0, -3);
		\coordinate (O3) at (9, -3);
		\coordinate (O4) at (9, 0);
		\coordinate (O5) at (3, 0);
		\coordinate (O6) at (3, 3);
		\coordinate (O7) at (0, 3);
		\coordinate (O8) at (3, -3);
		\coordinate (x0) at (1.5, 2.5);
		\coordinate (xf) at (8.5, -1.5);
		\coordinate (quad) at (1.5, -2);
		\coordinate (mid1) at (2, 0.5);
		\coordinate (mid2) at (5, -1);
		
		\fill[red!10] (O2) -- (O8) -- (O6) --(O7) --cycle;
		\node[label={[red, rotate=-90]below:\scriptsize corridor 1}] at (1, 0.5) {};

		\fill[blue!10] (O8) -- (O3) -- (O4) --(O5) --cycle;
		\node[label={[blue]below:\scriptsize corridor 2}] at (6, -2) {};
		
		\draw (O7) -- (O2) -- (O3);
		\draw (O4) -- (O5) -- (O6);
		
		\pattern[pattern=north east lines] ($(O7)+(-0.2, 0)$) rectangle ($(O2)+(0, -0.2)$);
		
		\pattern[pattern=north east lines] ($(O2)+(0, -0.2)$) rectangle (O3);

		\pattern[pattern=north east lines] ($(O4)+(0, 0.2)$) rectangle ($(O5)+(0, 0)$);
		
		\pattern[pattern=north east lines] ($(O5)+(0.2, 0.2)$) rectangle (O6);

		\fill (x0) circle [radius=2pt];
		\node[label=left:{\footnotesize $x_0$}] at (x0) {};
		
 		\fill (xf) circle [radius=2pt];
 		\node[label=below:{\footnotesize $x_{\tau}$}] at (xf) {};
		
		\node [quadcopter top,fill=white,draw=black,minimum width=1cm,rotate=10,scale=0.5] at (quad) {}; 
		
		\draw[dashed, -latex] (x0) to[out=-85,in=90] (mid1);
		
		\draw[dashed, -latex] (mid1) to[out=-90,in=90] (quad) to[out=-80,in=180] (mid2);
		
		\draw[dashed] (mid2) to[out=0,in=180] (xf);

       	\end{tikzpicture}
		\caption{An illustration of path-plannning in an L-shaped corridor.}
		\label{fig: corridor}
\end{figure}
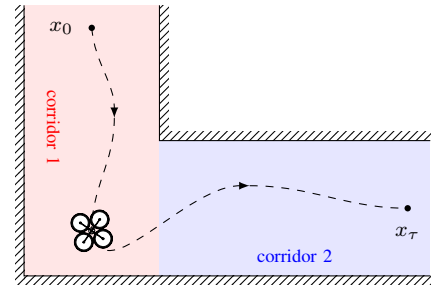

Optimization \eqref{opt: mixed} is challenging to solve if the number of binary variables is large. For example, if \(\tau=22\), then optimization \eqref{opt: mixed} contains \(21\) binary variables, and solving it naively requires considering \(2^{21}\) (more than 2 millions) different values of vector \(b_{[0, \tau-1]}\), each value corresponds to a convex optimization problem.

One can reduce the number of binary variables in optimization \eqref{opt: mixed} using infeasibility detection as follows. Let \(\hat{b}\in\{0, 1\}\) and \((u^\star_{[0, \tau-1]}, x^\star_{[1, \tau-1]}, b^\star_{[1, \tau-1]})\) be an optimal solution of optimization \eqref{opt: mixed}. If \(b^\star_1=\hat{b}\), then \((u^\star_{[0, \tau-1]}, x^\star_{[1, \tau-1]}, b^\star_{[1, \tau-1]})\) satisfies the constraints of the following optimization:
\begin{equation}\label{opt: BnB}
    \begin{array}{ll}
        \underset{\substack{u_{[0, \tau-1]}, x_{[1, \tau-1]}\\ b_{[1, \tau-1]}} }{\mbox{minimize}}  & \frac{1}{2}\sum_{t=0}^\tau \norm{u_t}^2\\
        \quad\mbox{subject to} &  x_{t+1}=Ax_t+Bu_t+h,\, 0\leq t\leq \tau-1,\\
        & u_t\in\mathbb{U},\enskip 0\leq t\leq \tau-1,\\
        & x_{t}\in\mathbb{X},\enskip 1\leq t\leq \tau-1,\\
        & \underline{r}^1+b_t(\underline{r}^2-\underline{r}^1)\leq Tx_t,1\leq t\leq \tau-1,\\
        & \overline{r}^1+b_t(\overline{r}^2-\overline{r}^1)\geq Tx_t,1\leq t\leq \tau-1,\\
        & b_1=\hat{b},\,b_t\in[0, 1], \, 1\leq t\leq \tau-1.
    \end{array}
\end{equation}
Compared with \eqref{opt: mixed}, here the constraints on \(b_{[1, \tau-1]}\) are continuous instead of discrete, and the value of \(b_1\) is fixed instead of being a variable.

If optimization \eqref{opt: BnB} is infeasible (\ie, no solution satisfies its constraints), then \((u^\star_{[0, \tau-1]}, x^\star_{[1, \tau-1]}, b^\star_{[1, \tau-1]})\) does not satisfy the constraints in \eqref{opt: BnB}. Consequently, one must have \(b^\star_1\neq\hat{b}\). In other words, we can eliminate the binary variable \(b_1\) in optimization \eqref{opt: mixed} by fixing its value to be \(1-\hat{b}\).

In the following, we consider optimization \eqref{opt: mixed} where \(\tau=22\) and
\begin{equation*}
    \begin{aligned}
     &x_0=\begin{bmatrix}
1 & 9 & 2.5 & 0 & 0 & 0
\end{bmatrix}^\top,\\
&x_\tau=\begin{bmatrix}
12 & -1 & 0.5 & 0 & 0 & 0
\end{bmatrix}^\top.
    \end{aligned}
\end{equation*}.

To eliminate binary variable \(b_1\) in the above optimization, we apply Algorithm~\ref{alg: splitting} and Algorithm~\ref{alg: PIPG} to the corresponding optimization \eqref{opt: BnB}. Particularly, Appendix~\ref{app: transform} gives the transformation from optimization~\eqref{opt: BnB} to optimization \eqref{opt: conic}.
Fig.~\ref{fig: mixed exp} illustrates the convergence of \(\norm{w^j-w^{j-1}}\) and \(\norm{z^j-z^{j-1}}\) computed in Algorithm~\ref{alg: splitting}  and Algorithm~\ref{alg: PIPG}, where the two algorithms show similar convergence. These convergence results show that both \(\norm{w^j-w^{j-1}}\) and \(\norm{z^j-z^{j-1}}\) converge to zero if \(\hat{b}=0\), and \(\norm{w^j-w^{j-1}}\) does not converge to zero if \(\hat{b}=1\). As a result of Theorem~\ref{thm: infeasibility}, we conclude that optimization \eqref{opt: BnB} is infeasible if \(\hat{b}=1\) and feasible if \(\hat{b}=0\). Therefore we can eliminate the binary variable \(b_1\) in optimization \eqref{opt: mixed} by fixing its value to be \(0\).

By repeating a similar process, we can further reduce the number of binary variables, and conclude that \(b_i=0\) for all \(i=1, 2, \ldots, 7\) and \(b_i=1\) for all \(i=12, 13, \ldots, 21\) in optimization \eqref{opt: mixed}. In other words, by solving at most \(42\) convex optimization problems, each similar to optimization \eqref{opt: BnB}, we can reduce the number of possible values of binary variables \(b_{[1, 21]}\) from \(2^{21}\) (more than 2 millions) to merely \(2^4\) (less than 20)!

\begin{figure}[!ht]
    \centering
    \begin{subfigure}[b]{0.49\columnwidth}
    \centering 
    \includegraphics[width=\textwidth]{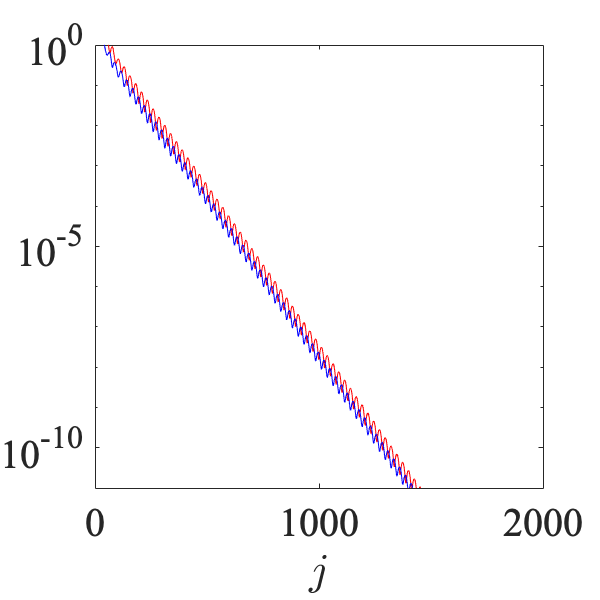}
    \caption{\(\hat{b}=0\), Algorithm~\ref{alg: splitting} }
    \end{subfigure}
    \begin{subfigure}[b]{0.49\columnwidth}
    \centering 
    \includegraphics[width=\textwidth]{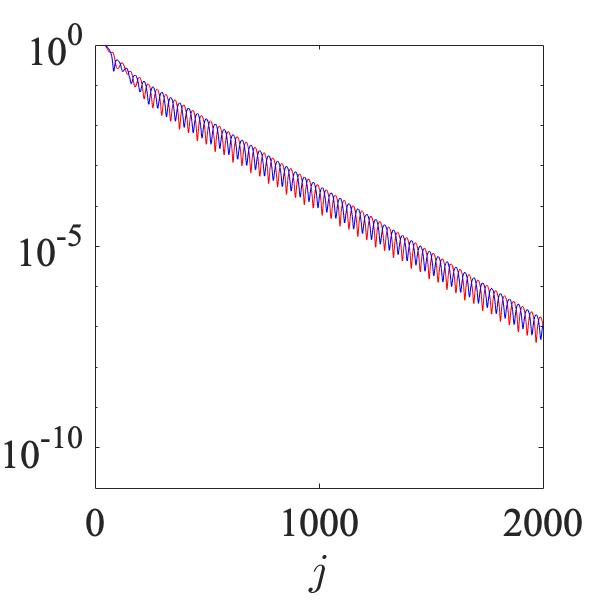}
    \caption{\(\hat{b}=0\), Algorithm~\ref{alg: PIPG} }
    \end{subfigure}
    \begin{subfigure}[b]{0.49\columnwidth}
    \centering 
    \includegraphics[width=\textwidth]{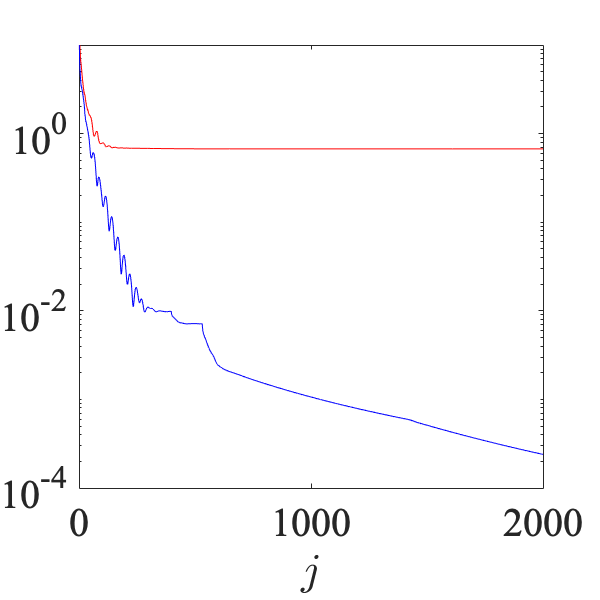}
    \caption{\(\hat{b}=1\), Algorithm~\ref{alg: splitting}}
    \end{subfigure}
    \begin{subfigure}[b]{0.49\columnwidth}
    \centering 
    \includegraphics[width=\textwidth]{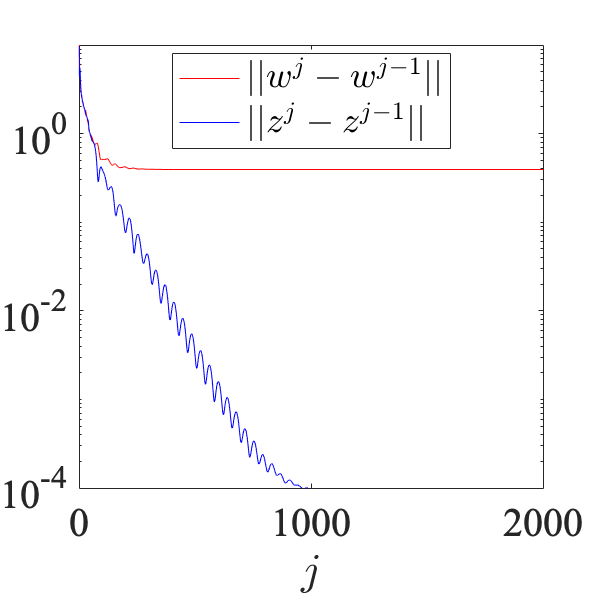}
    \caption{\(\hat{b}=1\), Algorithm~\ref{alg: PIPG}}
    \end{subfigure}
    \caption{Convergence of \(\norm{w^j-w^{j-1}}\) and \(\norm{z^j-z^{j-1}}\) in Algorithm~\ref{alg: splitting} and Algorithm~\ref{alg: PIPG} when applied to optimization \eqref{opt: BnB} for different values of \(\hat{b}\).} 
    \label{fig: mixed exp}
\end{figure}


\section{Conclusions}
\label{sec: conclusions}
We introduce a novel method, named proportional-integral projected gradient method (PIPG), for infeasibility detection in conic optimization. The iterates of PIPG either asymptotically provide a proof of primal or dual infeasibility, or asymptotically satisfy a set of primal-dual optimality conditions. We demonstrate the application of PIPG in quasiconvex and mixed integer optimization. 

There are several questions about PIPG that still remain open. For example, does PIPG allow larger step sizes like the primal-dual method in \cite{li2021new}? If the projections in PIPG are computed approximatedly rather than exactly, how will the results in Theorem~\ref{thm: infeasibility} change? We aim to answer these questions in our future work.



\appendices

\section{Proof of Proposition~\ref{prop: PIPG avg op}}
\label{app: lemA}
We will use the following result on projection.

\begin{lemma}\cite[Thm. 3.16]{bauschke2017convex}\label{lem: projection}
Let \(\mathbb{D}\subset\mathbb{R}^n\) be a nonempty closed convex set. Then \(y=\pi_{\mathbb{D}}[z]\) if and only if \(y\in\mathbb{D}\) and 
\(\langle y-z, z'-y\rangle\geq 0\) for any \(z'\in\mathbb{D}\).
\end{lemma}

\begin{proof}
Let \(z_j\in\mathbb{D}\), \(v_j\in\mathbb{R}^m\) and
\begin{subequations}\label{eqn: two pts}
\begin{align}
    w^+_j\coloneqq&\pi_{\mathbb{K}^\circ}[v_j+\alpha(Hz_j-g)],\label{eqn: w proj}\\
    z^+_j\coloneqq&\pi_{\mathbb{D}}[z_j-\alpha( Pz_j+q+ H^\top w^+_j)],\label{eqn: z proj}\\
   v^+_j\coloneqq&w_j^++\alpha H(z^+_j-z_j),\label{eqn: v int}
\end{align}
\end{subequations}
for \(j=1, 2\). Inorder to show that line~\ref{alg: w}-\ref{alg: v} in Algorithm~\ref{alg: PIPG} is the fixed point iteration of a \(\gamma\)-averaged operator, it suffices to show the following inequality
\begin{equation}\label{eqn: zw avg op}
\begin{aligned}
      & \textstyle \norm{z_2^+-z_1^+}^2+\frac{1-\gamma}{\gamma}\norm{z_1-z_1^+-z_2+z_2^+}^2\\
      &\textstyle +\norm{v_2^+-v_1^+}^2+\frac{1-\gamma}{\gamma}\norm{v_1-v_1^+-v_2+v_2^+}^2\\
      &\leq \norm{z_2-z_1}^2+\norm{v_2-v_1}^2
\end{aligned}
\end{equation}
To this end, first we apply Lemma~\ref{lem: projection} to \eqref{eqn: w proj} and \eqref{eqn: z proj} for both \(j=1\) and \(j=2\) and obtain the following inequalities:
\begin{subequations}
    \begin{align}
        0\leq & \langle w_1^+-v_1-\alpha (Hz_1-g), w_2^+-w_1^+\rangle,\label{eqn: w1 normal}\\
        0\leq & \langle w_2^+-v_2-\alpha (Hz_2-g), w_1^+-w_2^+\rangle,\label{eqn: w2 normal}\\
        0\leq &\langle z_1^+-z_1+\alpha(Pz_1+q+ H^\top w_1^+), z_2^+-z_1^+\rangle,\label{eqn: z1 normal}\\
        0\leq &\langle z_2^+-z_2+\alpha (Pz_2+q+H^\top w_2^+), z_1^+-z_2^+\rangle.\label{eqn: z2 normal}
    \end{align}
\end{subequations}
Summing up the right hand side of \eqref{eqn: w1 normal}, \eqref{eqn: w2 normal}, \eqref{eqn: z1 normal}, and \eqref{eqn: z2 normal} gives the following
\begin{equation}\label{eqn: sum proj}
    \begin{aligned}
         &0\leq \langle z_1^+-z_2^+, z_1-z_1^+-z_2+z_2^+\rangle\\
         &+ \alpha\langle P(z_2-z_1), z_1^+-z_2^+ \rangle\\
    &+\langle w_1^+-w_2^+, w_2^+-v_2-w_1^++v^1\rangle\\
    &+\alpha\langle w_1^+-w_2^+, H(z_1-z_2-z_1^++z_2^+)\rangle.
    \end{aligned}
\end{equation}
Second, by using \eqref{eqn: v int} for \(j=1\) and \(j=2\) we can show the following two identities
\begin{equation}\label{eqn: wv eqn1}
    \begin{aligned}
         &\alpha\langle w_1^+-w_2^+, H(z_1-z_2-z_1^++z_2^+)\rangle\\
         &=\langle w_1^+-w_2^+, w_1^+-v_1^++v_2^+-w_2^+\rangle,
    \end{aligned}
\end{equation}
\begin{equation}\label{eqn: wv eqn2}
    \begin{aligned}
         &\langle w_1^+-w_2^+, v_1-v_1^+-v_2+v_2^+\rangle\\
         &=\langle v_1^+-v_2^+, v_1-v_1^+-v_2+v_2^+\rangle\\
         &+\alpha\langle H(z_1-z_1^+-z_2+z_2^+), v_1-v_1^+-v_2+v_2^+\rangle.
    \end{aligned}
\end{equation}
By summing up both sides of \eqref{eqn: sum proj}, \eqref{eqn: wv eqn1}, and \eqref{eqn: wv eqn2} we obtain the following
\begin{equation}\label{eqn: sum 2}
    \begin{aligned}
         &0\leq \langle z_1^+-z_2^+, z_1-z_1^+-z_2+z_2^+\rangle\\
         &+\langle v_1^+-v_2^+, v_1-v_1^+-v_2+v_2^+\rangle\\
         &+\alpha\langle P(z_2-z_1), z_1^+-z_2^+ \rangle\\
    &+\alpha\langle H(z_1-z_1^+-z_2+z_2^+), v_1-v_1^+-v_2+v_2^+\rangle.
    \end{aligned}
\end{equation}
Our next step is to further simplify the inner product terms in \eqref{eqn: sum 2} as follows. First, by completing the square, we can show the following inequality
\begin{equation}\label{eqn: zvH}
    \begin{aligned}
    &2\langle H(z_1-z_1^+-z_2+z_2^+), v_1-v_1^+-v_2+v_2^+\rangle\\
    &\leq \textstyle \frac{\gamma\alpha}{2\gamma -1}\norm{H(z_1-z_1^+-z_2+z_2^+)}^2\\
    &\textstyle+\frac{2\gamma-1}{\gamma\alpha} \norm{v_1-v_1^+-v_2+v_2^+}^2.
    \end{aligned}
\end{equation}
Second, by completing the square we can also show the following identities
\begin{equation}\label{eqn: z square}
\begin{aligned}
    &\norm{z_1-z_1^+-z_2+z_2^+}^2+\norm{z_1^+-z_2^+}^2-\norm{z_1-z_2}^2\\
    &=-2\langle z_1^+-z_2^+, z_1-z_1^+-z_2+z_2^+\rangle,
\end{aligned}
\end{equation}
\begin{equation}\label{eqn: v square}
\begin{aligned}
 &\norm{v_1-v_1^+-v_2+v_2^+}^2+\norm{v_1^+-v_2^+}^2-\norm{v_1-v_2}^2\\
 &=-2\langle v_1^+-v_2^+, v_1-v_1^+-v_2+v_2^+\rangle.
\end{aligned}
\end{equation}
Third, since \(\nu\geq \mnorm{H}\), we have
\begin{equation}\label{eqn: H sigma}
    \begin{aligned}
     \norm{H(z_1-z_1^+-z_2+z_2^+)}^2\leq \nu^2\norm{z_1-z_1^+-z_2+z_2^+}^2.
    \end{aligned}
\end{equation}
If \(\lambda=0\geq \mnorm{P}\), then we necessarily have \(P=0\). In this case, by multiplying both sides of \eqref{eqn: sum 2} by \(2\), both sides of \eqref{eqn: zvH} by \(\alpha\), both sides of \eqref{eqn: H sigma} by \(\frac{\gamma\alpha^2}{2\gamma-1}\), then summing them up together with the both sides of \eqref{eqn: z square}, \eqref{eqn: v square}, we obtain \eqref{eqn: zw avg op} if \(\frac{\gamma\alpha^2\nu^2}{2\gamma-1}\leq \frac{2\gamma-1}{\gamma}\). 

If \(\lambda>0\), we need the following two additional inequalities. First, since \(\lambda\geq \mnorm{P}\) one can verify the following
\begin{equation}\label{eqn: f smooth}
 -\langle P(z_1-z_2), z_1-z_2\rangle\leq \textstyle-\frac{1}{\lambda}\norm{P(z_1-z_2)}^2.
\end{equation}
Second, by completing the square we can show the following
\begin{equation}\label{eqn: f square}
    \begin{aligned}
    &2\langle P(z_2-z_1), z_1^+-z_1-z_2^++z_2 \rangle\\
    & \leq \textstyle\frac{2}{\lambda} \norm{P(z_2- z_1)}^2+\frac{\lambda}{2} \norm{z_1^+-z_1-z_2^++z_2}^2.
    \end{aligned}
\end{equation}
Finally, by multiplying both sides of \eqref{eqn: sum 2} by \(2\), both sides of \eqref{eqn: zvH} by \(\alpha\), both sides of \eqref{eqn: H sigma} by \(\frac{\gamma\alpha^2}{2\gamma-1}\), both sides of \eqref{eqn: f smooth} by \(2\alpha\), both sides of \eqref{eqn: f square} by \(\alpha\), then summing them up together with both sides of \eqref{eqn: z square} and \eqref{eqn: v square}, we obtain \eqref{eqn: zw avg op} if
\(\frac{\gamma\alpha^2\nu^2}{2\gamma-1}+\frac{\alpha\lambda}{2}\leq \frac{2\gamma-1}{\gamma}\). 

Combining the above two case, we conclude that \eqref{eqn: zw avg op} holds for all \(\lambda\in\mathbb{R}_+\) if
\[\textstyle \frac{\gamma\alpha^2\nu^2}{2\gamma-1}+\frac{\alpha\lambda}{2}\leq \frac{2\gamma-1}{\gamma}.\]
One can verify that the above inequality holds under Assumption~\ref{asp: step size}, which completes the proof.
\end{proof}

\section{Proof of Lemma~\ref{lem: dist func}}
\label{app: dist func}
We will again use Lemma~\ref{lem: projection}.

\begin{proof}

Using line~\ref{alg: w}-\ref{alg: v} in Algorithm~\ref{alg: PIPG} we can show the following
\begin{equation*}
    \begin{aligned}
    z^j=&\pi_{\mathbb{D}}[z^{j-1}-\alpha (Pz^{j-1}+q+H^\top w^j)],\\
    w^j=&\pi_{\mathbb{K}^\circ}[w^{j-1}+\alpha(H(2z^{j-1}-z^{j-2})-g)].
\end{aligned}
\end{equation*}
Applying Lemma~\ref{lem: projection} to the above two projections and using the definition of normal cone in \eqref{eqn: normal cone}  we can show
\begin{subequations}
\begin{align}
    &\textstyle \frac{1}{\alpha}(z^{j-1}-z^j)-Pz^{j-1}-q-H^\top w^j\in N_{\mathbb{D}}(z^j),\label{eqn: z normal}\\
    &\textstyle \frac{1}{\alpha}(w^{j-1}-w^j)+H(2z^{j-1}-z^j)-g\in N_{\mathbb{K}^\circ}(w^j).\label{eqn: w normal}
\end{align}
\end{subequations}
Hence
\begin{equation*}
    \begin{aligned}
        &d(-Pz^j-q-H^\top w^j|N_{\mathbb{D}}(z^j))\\
        &\textstyle \leq \norm{\frac{1}{\alpha}(z^j-z^{j-1})-Pz^j+Pz^{j-1}}\\
        &\textstyle \leq \frac{1}{\alpha}\norm{z^j-z^{j-1}}+\norm{P(z^j-z^{j-1})}
    \end{aligned}
\end{equation*}
where the first inequality is due to the definition of distance function in \eqref{eqn: distance func} and \eqref{eqn: z normal}, the second inequality is due to the triangle inequality. Finally, combining the above inequality with the fact that \(\lambda\norm{z}\geq\norm{Pz}\) for any \(z\in\mathbb{R}^n\), we obtain the first inequality in Lemma~\ref{lem: dist func}.

Similarly, we can show the following
\begin{equation*}
    \begin{aligned}
        &d(Hz^j-g|N_{\mathbb{K}^\circ}(w^j))\\
        &\textstyle \leq \norm{\frac{1}{\alpha}(w^j-w^{j-1})+H(z^j-2z^{j-1}+z^{j-2})}\\
        &\textstyle \leq \frac{1}{\alpha}\norm{w^j-w^{j-1}}+\norm{H(z^j-2z^{j-1}+z^{j-2})},
    \end{aligned}
\end{equation*} 
where the first inequality is due to the definition of distance function in \eqref{eqn: distance func} and \eqref{eqn: w normal}, and the second inequality is due to the triangle inequality. Finally, combining the above inequality with the fact that \(\nu\norm{z}\geq\norm{Hz}\) for any \(z\in\mathbb{R}^n\), we obtain the second inequality in Lemma~\ref{lem: dist func}.

\end{proof}
\section{Proof of Lemma~\ref{lem: separating} }
\label{app: lemB}
We will use Lemma~\ref{lem: Moreau}, Lemma~\ref{lem: projection} and the following result.

\begin{lemma}\cite[Lem. 3.2]{banjac2021asymptotic}\label{lem: seq proj}
Let \(\{y^j\}_{j\in\mathbb{N}}\) be a sequence such that \(y^j\in\mathbb{R}^n\) for all \(j\in\mathbb{N}\) and there exists \(\overline{y}\in\mathbb{R}^n\) such that \(\lim_{j\to\infty} \frac{y^j}{j}=\overline{y}\). Let \(\mathbb{D}\subset\mathbb{R}^n\) be a closed convex set. 
\begin{enumerate}
    \item \(y^j-\pi_{\mathbb{D}}[y^j]\in (\rec \mathbb{D})^\circ\).
    \item \(\lim\limits_{j\to\infty}\frac{\pi_{\mathbb{D}}[y^j]}{j}=\pi_{\rec \mathbb{D}}[\overline{y}]\).
    \item \(\lim\limits_{j\to\infty}\frac{y^j-\pi_{\mathbb{D}}[y^j]}{j}=\pi_{(\rec \mathbb{D})^\circ}[\overline{y}]\).
    \item \(\lim\limits_{j\to\infty} \frac{\langle \pi_{\mathbb{D}}[y^j], y^j-\pi_{\mathbb{D}}[y^j]\rangle}{j}=\sigma_{\mathbb{D}}(\pi_{(\rec \mathbb{D})^\circ}[\overline{y}])\).
\end{enumerate}
\end{lemma}

We will also use the following fact: given sequences \(\{y^j\}_{j\in\mathbb{N}}\) and \(\{z^j\}_{j\in\mathbb{N}}\) where \(x^j, y^j\in\mathbb{R}^n\) for all \(j\), if the limits \(\lim_{j\to\infty} y^j\) and \(\lim_{j\to\infty} z^j\) both exists and are finite-valued, then
\begin{equation}\label{eqn: limit inner}
    \lim_{j\to\infty} \langle y^j, z^j\rangle=\langle \lim_{j\to\infty} y^j, \lim_{j\to\infty} z^j \rangle.
\end{equation}

\begin{proof}
We start by showing the two limits. From Proposition~\ref{prop: PIPG avg op} and Lemma~\ref{lem: min dis} we know that there exists \(\overline{z}\in\mathbb{R}^n\) and \(\overline{w}\in\mathbb{R}^m\) such that 
\begin{subequations}\label{eqn: zv min dis}
    \begin{align}
        &\lim_{j\to\infty}\textstyle \frac{z^j}{j}=\lim\limits_{j\to\infty} z^j-z^{j-1}=\overline{z},\label{eqn: z min dis}\\
        &\lim_{j\to\infty}\textstyle \frac{v^j}{j}=\lim\limits_{j\to\infty}v^j-v^{j-1}=\overline{w},\label{eqn: v min dis}
    \end{align}
\end{subequations}
In addition, using line~\ref{alg: v} in Algorithm~\ref{alg: PIPG} one can show the following
\begin{equation*}
   \begin{aligned}
       v^j-v^{j-1}=&w^j-w^{j-1}+\alpha H(z^j-2z^{j-1}+z^{j-2}),\\
       \textstyle\frac{v^j}{j}=&\textstyle\frac{w^j}{j}+\frac{\alpha}{j} H(z^j-z^{j-1}),
   \end{aligned}
\end{equation*}
Letting \(j\to\infty\) in the above two equations, then using \eqref{eqn: z min dis} and \eqref{eqn: v min dis} we can show the following
\begin{equation}
    \lim_{j\to\infty} \textstyle\frac{w^j}{j}=\lim\limits_{j\to\infty}w^{j+1}-w^j=\overline{w}.\label{eqn: w min dis}
\end{equation}
The above equation and \eqref{eqn: z min dis} together give the two limits in Lemma~\ref{lem: separating}.

Next, we show that \(\overline{z}\in\rec\mathbb{D}\), \(\overline{w}\in\mathbb{K}^\circ\) and the conditions in \eqref{eqn: pd inf} hold. To this end, let
\begin{subequations}\label{eqn: zw hat}
\begin{align}
    \hat{z}^j\coloneqq &z^{j-1}-\alpha (Pz^{j-1}+q+H^\top w^j)\label{eqn: z hat}\\
    \hat{w}^j\coloneqq & w^{j-1}+\alpha (H(2z^{j-1}-z^{j-2})-g)\label{eqn: w hat}
\end{align}
\end{subequations}
Then using \eqref{eqn: z min dis} and \eqref{eqn: w min dis} we can show that the following two limits
\begin{subequations}
    \begin{align}
         & \lim_{j\to\infty}\textstyle \frac{\hat{z}^j}{j}= \overline{z}-\alpha (P\overline{z}+H^\top \overline{w})\\
         & \lim_{j\to\infty}\textstyle \frac{\hat{w}^j}{j}=\overline{w}+\alpha H\overline{z}
    \end{align}
\end{subequations}
Further, using line~\ref{alg: w}-\ref{alg: v} in Algorithm~\ref{alg: PIPG} we can show that
\begin{equation}\label{eqn: zw proj}
    z^j=\pi_{\mathbb{D}}[\hat{z}^j],\enskip w^j=\pi_{\mathbb{K}^\circ}[\hat{w}^j].
\end{equation}
Hence, by applying Lemma~\ref{lem: seq proj} to sequences \(\{\hat{z}^j\}_{j\in\mathbb{N}}\) and \(\{\hat{w}^j\}_{j\in\mathbb{N}}\) we can show that 
\begin{subequations}
\begin{align}
      &\hat{z}^j-z^j\in (\rec \mathbb{D})^\circ, \enskip \hat{w}^j-w^j\in \mathbb{K},\label{eqn: diff rec}\\
      &\lim_{j\to\infty} \textstyle \frac{z^j}{j}=\overline{z}=\pi_{\rec \mathbb{D}}[\overline{z}-\alpha (P\overline{z}+H^\top \overline{w})],\label{eqn: proj rec D}\\
      & \begin{aligned}
          \lim_{j\to\infty} \textstyle \frac{\hat{z}^j-z^j}{j}=&-\alpha (P\overline{z}+H^\top \overline{w})\\
          =&\pi_{(\rec \mathbb{D})^\circ}[\overline{z}-\alpha (P\overline{z}+H^\top \overline{w})],
      \end{aligned}\label{eqn: proj rec D polar}\\
      &\lim_{j\to\infty} \textstyle \frac{w^j}{j}=\overline{w}=\pi_{\mathbb{K}^\circ}[\overline{w}+\alpha H \overline{z}],\label{eqn: proj K polar}\\
      &\lim_{j\to\infty} \textstyle \frac{\hat{w}^j-w^j}{j}=\alpha H\overline{z}=\pi_{\mathbb{K}}[\overline{w}+\alpha H \overline{z}],\label{eqn: proj K}\\
      &\lim_{j\to\infty} \textstyle \frac{\langle z^j, \hat{z}^j-z^j\rangle}{j}=\sigma_{\mathbb{D}}(-\alpha (P\overline{z}+H^\top \overline{w})),\label{eqn: sup D}\\
      &\lim_{j\to\infty} \textstyle \frac{\langle w^j, \hat{w}^j-w^j\rangle}{j}=\sigma_{\mathbb{K}^\circ}(\alpha H\overline{z})=0,\label{eqn: sup K polar}
\end{align}
\end{subequations}
where we used \eqref{eqn: z min dis}, \eqref{eqn: w min dis}, and the fact that \(\rec \mathbb{K}^\circ=\mathbb{K}^\circ\) and \((\mathbb{K}^\circ)^\circ=\mathbb{K}\). Further, \eqref{eqn: sup K polar} is due to \eqref{eqn: proj K} and the definition of support function and polar cone in \eqref{eqn: support} and \eqref{eqn: polar cone}, respectively. Notice that \eqref{eqn: proj rec D}, \eqref{eqn: proj K polar} and \eqref{eqn: proj K} imply the following
\begin{equation}\label{eqn: zw rec}
\overline{z}\in\rec \mathbb{D}, \enskip \overline{w}\in\mathbb{K}^\circ, \enskip H\overline{z}\in\mathbb{K}.
\end{equation}
Further, by applying Lemma~\ref{lem: Moreau} to the projections in \eqref{eqn: proj rec D}, \eqref{eqn: proj rec D polar}, \eqref{eqn: proj K polar} and \eqref{eqn: proj K}, we obtain the following
\begin{equation}\label{eqn: complementary}
    \langle \overline{z}, P\overline{z}+H^\top \overline{w}\rangle=0,\enskip \langle \overline{w}, H\overline{z}\rangle=0.
\end{equation}
Combining the above two equalities and the fact that \(P\) is symmetric and positive semidefinite, we obtain the following 
\begin{equation}\label{eqn: Pz=0}
    P\overline{z}=0
\end{equation}

Equipped with \eqref{eqn: zw rec} and \eqref{eqn: Pz=0}, our emaining task is to prove the last two equalities in \eqref{eqn: pd inf}. To this end, using  \eqref{eqn: diff rec}, \eqref{eqn: proj rec D}, \eqref{eqn: proj K polar} and the definition of polar cone we can show that
\begin{equation}\label{eqn: polar inner}
    \begin{aligned}
        &\langle z^{j-1}-z^j-\alpha (q+H^\top w^j), \overline{z}\rangle\leq 0, \\
        &\langle w^{j-1}-w^j+\alpha (H(2z^{j-1}-z^{j-2})-g), \overline{w}\rangle\leq 0,
    \end{aligned}
\end{equation}
where we used the definition of \(\hat{z}^j\) and \(\hat{w}^j\) in \eqref{eqn: zw hat}, and the equality in \eqref{eqn: Pz=0}. By multiplying both inequalities in \eqref{eqn: polar inner} by \(\frac{1}{\alpha j}\) and letting \(j\to\infty\) we obtain the following
\begin{subequations}\label{eqn: two bounds}
    \begin{align}
        &\textstyle \frac{1}{\alpha}\norm{\overline{z}}^2+\langle q, \overline{z}\rangle\geq -\lim\limits_{j\to\infty} \langle w^j, H\overline{z}\rangle\geq 0, \label{eqn: two bounds a}\\
        & \textstyle \frac{1}{\alpha}\norm{\overline{w}}^2+ \langle g, \overline{w}\rangle\geq \lim\limits_{j\to\infty}\langle z^{j-1}, H^\top \overline{w}\rangle\geq - \sigma_{\mathbb{D}}(-H^\top \overline{w}),\label{eqn: two bounds b}
    \end{align}
\end{subequations}
where the last step in \eqref{eqn: two bounds a} is due to \eqref{eqn: zw proj}, \eqref{eqn: proj K}, and the definition of polar cone in \eqref{eqn: polar cone}. Further, the last step in \eqref{eqn: two bounds b} is due to the definition of support function in \eqref{eqn: support}. Summing up the two inequalities in \eqref{eqn: two bounds} we obtain the following
\begin{equation}\label{eqn: ineq 1}
    \textstyle \sigma_{\mathbb{D}}(-H^\top \overline{w})+\langle q, \overline{z}\rangle+\langle \overline{w}, g\rangle+\frac{1}{\alpha}\norm{\overline{z}}+\frac{1}{\alpha}\norm{\overline{w}}^2\geq 0.
\end{equation}

We will show that the above inequality actually holds as an equality as follows.
First, by using \eqref{eqn: limit inner}, \eqref{eqn: z min dis}, and \eqref{eqn: proj rec D polar} we can show the following
\begin{equation}\label{eqn: z upper}
    \textstyle \lim\limits_{j\to\infty}\frac{1}{j}\langle \hat{z}^j-z^j, z^j-z^{j-1}\rangle=\langle -\alpha (P\overline{z}+H^\top \overline{w}), \overline{z}\rangle\leq 0,
\end{equation}
where the last step is due to  \eqref{eqn: proj rec D polar} and the definition of polar cone in \eqref{eqn: polar cone}. In addition, by applying Lemma~\ref{lem: projection} to the first projection in \eqref{eqn: zw proj} we can show
\begin{equation}\label{eqn: z lower}
    \langle \hat{z}^j-z^j, z^j-z^{j-1}\rangle\geq 0
\end{equation}
Combining \eqref{eqn: z upper} and \eqref{eqn: z lower} gives the following
\begin{equation}\label{eqn: z inner zero}
    \textstyle 0=\lim\limits_{j\to\infty}\frac{1}{j}\langle \hat{z}^j-z^j, z^j-z^{j-1}\rangle.
\end{equation}
Second,  one can show the following
\begin{equation}\label{eqn: z inner}
    \begin{aligned}
        &\lim_{j\to\infty}\textstyle \frac{1}{j}\langle z^j-\hat{z}^j, z^{j-1}\rangle\\
        &=\lim_{j\to\infty}\textstyle \langle z^j-z^{j-1}+\alpha(Pz^{j-1}+q+H^\top w^j), \frac{z^{j-1}}{j}\rangle\\
        &=\textstyle\norm{\overline{z}}^2+\alpha\langle q, \overline{z}\rangle+\alpha \lim\limits_{j\to\infty} \frac{1}{j}\langle Pz^{j-1}+H^\top w^j, z^{j-1}\rangle\\
        &\geq\textstyle \norm{\overline{z}}^2+\alpha\langle q, \overline{z}\rangle+\alpha \lim\limits_{j\to\infty} \frac{1}{j}\langle H^\top w^j, z^{j-1}\rangle,
    \end{aligned}
\end{equation}
where the first step is due to \eqref{eqn: z hat}, the second step is due to \eqref{eqn: limit inner} and \eqref{eqn: z min dis} and the last step is due to the positive semidefiniteness of matrix \(P\). Third, we can show that following
\begin{equation}\label{eqn: w inner}
\begin{aligned}
    &0=\lim_{j\to\infty}\textstyle \frac{\langle w^j-\hat{w}^j, w^j\rangle}{j}\\
    &=\lim_{j\to\infty}\textstyle \langle w^j-w^{j-1}-\alpha(H(2z^{j-1}-z^{j-2})-g), \frac{w^j}{j}\rangle\\
    &=\norm{w}^2+\langle \overline{w}, g\rangle-\langle H \overline{z}, \overline{w}\rangle-\alpha \lim_{j\to\infty}\textstyle \frac{1}{j}\langle H^\top w^j, z^{j-1}\rangle,
\end{aligned}
\end{equation}
where the first equality is due to \eqref{eqn: sup K polar}, the second equality is due to \eqref{eqn: w hat}, and the last equality is due to \eqref{eqn: limit inner}, \eqref{eqn: z min dis} and \eqref{eqn: w min dis}.
By summing up both sides of \eqref{eqn: sup D}, \eqref{eqn: z inner zero}, \eqref{eqn: z inner} and \eqref{eqn: w inner} then multiplying the resulting inequality by \(\frac{1}{\alpha}\), we obtain the following
\begin{equation*}
    \textstyle \sigma_{\mathbb{D}}(-H^\top \overline{w})+\langle q, \overline{z}\rangle+\langle \overline{w}, g\rangle+\frac{1}{\alpha}\norm{\overline{z}}+\frac{1}{\alpha}\norm{\overline{w}}^2\leq 0,
\end{equation*}
where we also used \eqref{eqn: complementary}, \eqref{eqn: Pz=0}, and the fact that \(\alpha \sigma_{\mathbb{D}}(-H^\top\overline{w})=\sigma_{\mathbb{D}}(-\alpha H^\top\overline{w})\) when \(\alpha>0\). Combining the above inequality with the one in \eqref{eqn: ineq 1}, we conclude that the inequality in \eqref{eqn: ineq 1} holds as an equality. As a result, both inequalities in \eqref{eqn: two bounds} also hold as equalities, which completes the proof.
\end{proof}

\section{Transformation from optimal control problems to conic optimization problems}
\label{app: transform}

We let \(\mathbf{1}_n\) and \(\mathbf{0}_n\) denote the \(n\)-dimensional vectors of all \(1\)'s and all \(0\)'s, respectively. Let \(\diag (c)\) denote the diagonal matrix whose diagonal elements are given by vector \(c\). Further, we let
\begin{equation*}
    \begin{aligned}
    &\overline{A}=\begin{bmatrix}
     I_{6(\tau-1)}\\
     0_{6\times 6(\tau-1)}
    \end{bmatrix}-\begin{bmatrix}
    0_{6\times 6(\tau-1)}\\
    I_{\tau-1} \otimes A 
    \end{bmatrix},\\
    &\overline{B}=-I_\tau\otimes
    B,\enskip \overline{C}=I_\tau\otimes \begin{bmatrix}
    0 & 0 & 1
    \end{bmatrix} ,\\
    & \overline{D}=I_{\tau-1}\otimes \begin{bmatrix}
    I_3 & 0_{3\times 3}\\
    -I_3 & 0_{3\times 3}
    \end{bmatrix},\enskip \overline{E}=I_{\tau-1}\otimes \begin{bmatrix}
    \underline{r}_1-\underline{r}_2\\
     \overline{r}_2-\overline{r}_1
    \end{bmatrix},
    \end{aligned}
\end{equation*}
and 
\begin{equation*}
\begin{aligned}
    & H_1 = \begin{bmatrix}
    \overline{A} & \overline{B} \\
    0_{\tau\times 6(\tau-1)} & \overline{C}
    \end{bmatrix},\\
    &g_1 = \begin{bsmallmatrix}
    (Ax_0+h)^\top & (\mathbf{1}_{\tau-2}\otimes h)^\top & (-x_\tau+h)^\top & \rho_1\mathbf{1}_\tau^\top
    \end{bsmallmatrix}^\top,\\
    &H_2 = \begin{bmatrix}
    \overline{D} & 0_{6(\tau-1)\times 3\tau} & \overline{E}
    \end{bmatrix},\enskip g_2=\mathbf{1}_{\tau-1}\otimes\begin{bmatrix}
    \underline{r}_1\\
    -\overline{r}_1
    \end{bmatrix}\\
    & P_1 =\diag \left(\begin{bmatrix}\mathbf{0}_{6(\tau-1)}^\top & \mathbf{1}_{3\tau}^\top  \end{bmatrix}^\top\right), q_1=\mathbf{0}_{9\tau-6},\\
    & P_2 =\diag \left(\begin{bmatrix}\mathbf{0}_{6(\tau-1)}^\top & \mathbf{1}_{3\tau}^\top & \mathbf{0}_{\tau-1}^\top \end{bmatrix}^\top\right), q_2=\mathbf{0}_{10\tau-7}.
\end{aligned}
\end{equation*}
Further, we let 
\begin{equation*}
    \begin{aligned}
    \mathbb{X}_1 =&\{r\in\mathbb{R}^3| \norm{r}\cos\beta\leq [r]_2\}\times \{r\in\mathbb{R}^3|\norm{r}\leq \eta\},\\
    \mathbb{X}_2 =&\mathbb{R}^3\times \{r\in\mathbb{R}^3|\norm{r}\leq \eta\},\\
    \mathbb{U}_1=&\{u\in\mathbb{R}^{3}|\norm{u}\cos\theta\leq [u]_3,\norm{u}\leq \rho_2\},\\
     \mathbb{K}_1=&\{\mathbf{0}_{6\tau}\}\times \mathbb{R}_+^{\tau}.
    \end{aligned}
\end{equation*}
With the above definitions, we are ready to transform the optimal control problems in Section~\ref{sec: experiment} into space cases of conic optimization \eqref{opt: conic}. Optimization \eqref{opt: landing} is the special case of \eqref{opt: conic} where
\begin{equation*}
    \begin{aligned}
    &z = \begin{bmatrix}
    x_1^\top & \ldots x_{\tau-1}^\top & u_0^\top & \ldots & u_{\tau-1}^\top
    \end{bmatrix}^\top,\\
    & P=P_1, \enskip q=q_1, \enskip H=S_1H_1, \enskip g=S_1g_1,\enskip \mathbb{K}=\mathbb{K}_1,\\ 
    &\mathbb{D}= (\mathbb{X}_1)^{j-1}\times \{x_\tau\}^{\tau-j}\times (\mathbb{U}_1)^\tau,
    \end{aligned}
\end{equation*}
where \(S_1\) is a diagonal matrix with positive diagonal entries such that the rows in matrix \(H\) have unit \(\ell_2\) norm. And optimization \eqref{opt: landing} is the special case of \eqref{opt: conic} where
\begin{equation*}
    \begin{aligned}
    &z = \begin{bmatrix}
    x_1^\top & \ldots x_{\tau-1}^\top & u_0^\top & \ldots & u_{\tau-1}^\top & b_1 & \ldots b_{\tau-1}
    \end{bmatrix}^\top\\
    & H = S_2\begin{bmatrix}
    \begin{bmatrix}
    H_1 & 0_{7\tau\times (\tau-1)}
    \end{bmatrix}\\
    H_2
    \end{bmatrix}, \enskip g=S_2\begin{bmatrix}
    g_1\\
    g_2
    \end{bmatrix},\\
    &P=P_2, \enskip q=q_2,\enskip \mathbb{K}=\mathbb{K}_1\times \mathbb{R}_+^{6(\tau-1)},\\ &\mathbb{D}= (\mathbb{X}_2)^{\tau-1}\times (\mathbb{U}_1)^\tau\times \{1-\hat{b}\}\times [0, 1]^{\tau-2},
\end{aligned}
\end{equation*}
where \(S_2\) is a diagonal matrix with positive diagonal entries such that the rows in matrix \(H\) have unit \(\ell_2\) norm. Notice that, in both problems, we use matrix \(S_1\) and \(S_2\) to ensure the constraints parameters \(H\) and \(g\) are well-conditioned.


\bibliographystyle{IEEEtran}
\bibliography{IEEEabrv,reference}

\begin{thebibliography}{10}
\providecommand{\url}[1]{#1}
\csname url@rmstyle\endcsname
\providecommand{\newblock}{\relax}
\providecommand{\bibinfo}[2]{#2}
\providecommand\BIBentrySTDinterwordspacing{\spaceskip=0pt\relax}
\providecommand\BIBentryALTinterwordstretchfactor{4}
\providecommand\BIBentryALTinterwordspacing{\spaceskip=\fontdimen2\font plus
\BIBentryALTinterwordstretchfactor\fontdimen3\font minus
  \fontdimen4\font\relax}
\providecommand\BIBforeignlanguage[2]{{%
\expandafter\ifx\csname l@#1\endcsname\relax
\typeout{** WARNING: IEEEtran.bst: No hyphenation pattern has been}%
\typeout{** loaded for the language `#1'. Using the pattern for}%
\typeout{** the default language instead.}%
\else
\language=\csname l@#1\endcsname
\fi
#2}}

\bibitem{ben2001lectures}
A.~Ben-Tal and A.~Nemirovski, \emph{Lectures on Modern Convex Optimization:
  Analysis, Algorithms, and Engineering Applications}.\hskip 1em plus 0.5em
  minus 0.4em\relax SIAM, 2001.

\bibitem{boyd2004convex}
S.~P. Boyd and L.~Vandenberghe, \emph{Convex Optimization}.\hskip 1em plus
  0.5em minus 0.4em\relax Cambridge University Press, 2004.

\bibitem{banjac2019infeasibility}
G.~Banjac, P.~Goulart, B.~Stellato, and S.~Boyd, ``Infeasibility detection in
  the alternating direction method of multipliers for convex optimization,''
  \emph{J. Optim. Theory Appl.}, vol. 183, no.~2, pp. 490--519, 2019.

\bibitem{banjac2021minimal}
G.~Banjac, ``On the minimal displacement vector of the {D}ouglas--{R}achford
  operator,'' \emph{Ope. Res. Lett.}, vol.~49, no.~2, pp. 197--200, 2021.

\bibitem{barratt2021automatic}
S.~Barratt, G.~Angeris, and S.~Boyd, ``Automatic repair of convex optimization
  problems,'' \emph{Optim. Eng.}, vol.~22, pp. 247--259, 2021.

\bibitem{agrawal2020disciplined}
A.~Agrawal and S.~Boyd, ``Disciplined quasiconvex programming,'' \emph{Optim.
  Lett.}, vol.~14, no.~7, pp. 1643--1657, 2020.

\bibitem{conforti2014integer}
M.~Conforti, G.~Cornu{\'e}jols, G.~Zambelli, \emph{et~al.}, \emph{Integer
  programming}.\hskip 1em plus 0.5em minus 0.4em\relax Springer, 2014, vol.
  271.

\bibitem{wolsey2020integer}
L.~A. Wolsey, \emph{Integer programming}.\hskip 1em plus 0.5em minus
  0.4em\relax John Wiley \& Sons, 2020.

\bibitem{raghunathan2014infeasibility}
A.~U. Raghunathan and S.~Di~Cairano, ``Infeasibility detection in alternating
  direction method of multipliers for convex quadratic programs,'' in
  \emph{Proc. IEEE Conf. Decision Control}.\hskip 1em plus 0.5em minus
  0.4em\relax IEEE, 2014, pp. 5819--5824.

\bibitem{o2016conic}
B.~O’Donoghue, E.~Chu, N.~Parikh, and S.~Boyd, ``Conic optimization via
  operator splitting and homogeneous self-dual embedding,'' \emph{J. Optim.
  Theory Appl.}, vol. 169, no.~3, pp. 1042--1068, 2016.

\bibitem{liu2019new}
Y.~Liu, E.~K. Ryu, and W.~Yin, ``A new use of {D}ouglas--{R}achford splitting
  for identifying infeasible, unbounded, and pathological conic programs,''
  \emph{Math. Program.}, vol. 177, no.~1, pp. 225--253, 2019.

\bibitem{ye1994nl}
Y.~Ye, M.~J. Todd, and S.~Mizuno, ``An ${O}(\sqrt{n}{L})$-iteration homogeneous
  and self-dual linear programming algorithm,'' \emph{Math. Oper. Res.},
  vol.~19, no.~1, pp. 53--67, 1994.

\bibitem{nesterov1999infeasible}
Y.~Nesterov, M.~J. Todd, and Y.~Ye, ``Infeasible-start primal-dual methods and
  infeasibility detectors for nonlinear programming problems,'' \emph{Math.
  Program.}, vol.~84, no.~2, pp. 227--268, 1999.

\bibitem{malyuta2021convex}
D.~Malyuta, T.~P. Reynolds, M.~Szmuk, T.~Lew, R.~Bonalli, M.~Pavone, and
  B.~Acikmese, ``Convex optimization for trajectory generation,'' \emph{arXiv
  preprint arXiv:2106.09125 [math.OC]}, 2021.

\bibitem{malyuta2021advances}
D.~Malyuta, Y.~Yu, P.~Elango, and B.~A{\c{c}}ikme{\c{s}}e, ``Advances in
  trajectory optimization for space vehicle control,'' \emph{arXiv preprint
  arXiv:2108.02335 [math.OC]}, 2021.

\bibitem{malitsky2017primal}
Y.~Malitsky, ``The primal-dual hybrid gradient method reduces to a primal
  method for linearly constrained optimization problems,'' \emph{arXiv preprint
  arXiv:1706.02602 [math.OC]}, 2017.

\bibitem{applegate2021infeasibility}
D.~Applegate, M.~D{\'\i}az, H.~Lu, and M.~Lubin, ``Infeasibility detection with
  primal-dual hybrid gradient for large-scale linear programming,'' \emph{arXiv
  preprint arXiv:2102.04592 [math.OC]}, 2021.

\bibitem{bauschke2017convex}
H.~H. Bauschke and P.~L. Combettes, \emph{Convex analysis and monotone operator
  theory in Hilbert spaces}.\hskip 1em plus 0.5em minus 0.4em\relax Springer,
  2017, vol. 408.

\bibitem{rockafellar2009variational}
R.~T. Rockafellar and R.~J.-B. Wets, \emph{Variational analysis}.\hskip 1em
  plus 0.5em minus 0.4em\relax Springer Science \& Business Media, 2009, vol.
  317.

\bibitem{giselsson2016linear}
P.~Giselsson and S.~Boyd, ``Linear convergence and metric selection for
  {D}ouglas-{R}achford splitting and {ADMM},'' \emph{IEEE Trans. Autom.
  Control}, vol.~62, no.~2, pp. 532--544, 2016.

\bibitem{banjac2021asymptotic}
G.~Banjac and J.~Lygeros, ``On the asymptotic behavior of the
  {D}ouglas--{R}achford and proximal-point algorithms for convex
  optimization,'' \emph{Optim. Lett.}, pp. 1--14, 2021.

\bibitem{yu2020proportional}
Y.~Yu, P.~Elango, and B.~A{\c{c}}{\i}kme{\c{s}}e, ``Proportional-integral
  projected gradient method for model predictive control,'' \emph{IEEE Control
  Syst. Lett.}, 2020.

\bibitem{yu2021proportional}
Y.~Yu, P.~Elango, U.~Topcu, and B.~A{\c{c}}ikmeșe, ``Proportional-integral
  projected gradient method for conic optimization,'' \emph{arXiv preprint
  arXiv:2108.10260 [math.OC]}, 2021.

\bibitem{wang2010control}
J.~Wang and N.~Elia, ``Control approach to distributed optimization,'' in
  \emph{Proc. Allerton Conf. Commun. Control Comput.}\hskip 1em plus 0.5em
  minus 0.4em\relax IEEE, 2010, pp. 557--561.

\bibitem{yu2020mass}
Y.~Yu, B.~A{\c{c}}{\i}kme{\c{s}}e, and M.~Mesbahi, ``Mass--spring--damper
  networks for distributed optimization in non-{E}uclidean spaces,''
  \emph{Automatica}, vol. 112, p. 108703, 2020.

\bibitem{yu2020rlc}
Y.~Yu and B.~A{\c{c}}{\i}kme{\c{s}}e, ``{RLC} circuits-based distributed mirror
  descent method,'' \emph{IEEE Control Syst. Lett.}, vol.~4, no.~3, pp.
  548--553, 2020.

\bibitem{kuczynski1992estimating}
J.~Kuczy{\'n}ski and H.~Wo{\'z}niakowski, ``Estimating the largest eigenvalue
  by the power and lanczos algorithms with a random start,'' \emph{SIAM J
  Matrix Anal. Appl.}, vol.~13, no.~4, pp. 1094--1122, 1992.

\bibitem{szmuk2019successive}
M.~Szmuk, ``Successive convexification \& high performance feedback control for
  agile flight,'' Ph.D. dissertation, Dept. Aeronatu. \& Astronaut., Univ.
  Washington, 2019.

\bibitem{zhang2014minimum}
X.~Zhang, Y.~Fang, and N.~Sun, ``Minimum-time trajectory planning for
  underactuated overhead crane systems with state and control constraints,''
  \emph{IEEE Trans. Ind. Electron.}, vol.~61, no.~12, pp. 6915--6925, 2014.

\bibitem{li2021new}
Z.~Li and M.~Yan, ``New convergence analysis of a primal-dual algorithm with
  large stepsizes,'' \emph{Adv. Comput. Math.}, vol.~47, no.~1, pp. 1--20,
  2021.

\end{thebibliography}

\end{document}